%% file: scharlau-univariate.tex
\documentclass[12pt]{amsart}
\usepackage{amsfonts,amsthm,amsmath,amscd,amssymb}
\usepackage[T1]{fontenc}
\usepackage[mathscr]{eucal}
\usepackage{indentfirst}
\usepackage{graphicx}
\usepackage{graphics}
\usepackage{xcolor}
\usepackage{pict2e}
\usepackage{epic}
\usepackage{hyperref}
\usepackage{mathtools}
\numberwithin{equation}{section}
\usepackage[margin=2.9cm]{geometry}
\usepackage{epstopdf} 
\usepackage{float}
\usepackage{cite}
\usepackage{parskip}
\usepackage{array}
\usepackage{booktabs}
\usepackage{enumerate}
\usepackage{adjustbox}
\usepackage{diagbox}
\usepackage{xparse}
\usepackage{tikz}
\usetikzlibrary{matrix,backgrounds}
\pgfdeclarelayer{myback}
\pgfsetlayers{myback,background,main}

\NewDocumentCommand{\fhighlight}{O{blue!40} m m}{%
\fill[#1,opacity=.2] ([shift=({2pt,-2pt})]#2.north west) rectangle ([shift=({-2pt,2pt})]#3.south east);
}

\usepackage{multirow}
\usepackage{todonotes}
\usepackage{array}   
\newcolumntype{L}{>{$}l<{$}} 
\newcolumntype{C}{>{$}c<{$}} 
\input{commands.tex}
\input{autoref}

\newcommand{\mb}[1]{\mathbb{#1}}
\newcommand{\mc}[1]{\mathcal{#1}}
\newcommand{\mf}[1]{\mathfrak{#1}}

\newcommand{\floor}[1]{\lfloor{#1}\rfloor}

\newcommand{\Char}{\operatorname{char}}

\newcommand{\pcoor}[1]{%
  \begingroup\lccode`~=`: \lowercase{\endgroup
  \edef~}{\mathbin{\mathchar\the\mathcode`:}\nobreak}%
  [
  \begingroup
  \mathcode`:=\string"8000
  #1%
  \endgroup 
  ]
}

\providecommand{\del}{\partial}
\providecommand{\Pic}{\operatorname{Pic}}
\providecommand{\MW}{\mathrm{MW}}
\providecommand{\Th}{\operatorname{Th}}

\providecommand{\Hor}{\operatorname{Hor}}

\providecommand{\Bez}{\operatorname{B\acute ez}}
\providecommand{\GW}{\operatorname{GW}}

\providecommand{\A}{\mathbb{A}}
\renewcommand{\P}{\mathbb{P}}
\renewcommand{\O}{\mathcal{O}}
\let\til\widetilde
\let\smashprod\wedge

\newcommand{\fc}{f_{\mathfrak{c}}}
\newcommand{\fg}{f_{\mathfrak{g}}}

\makeatletter
\@namedef{subjclassname@2020}{\textup{2020} Mathematics Subject Classification}
\makeatother

\begin{document}
\title{Lifts, transfers, and degrees of univariate maps}

\author{Thomas Brazelton}
\address{Department of Mathematics, University of Pennsylvania} 
\email{tbraz@math.upenn.edu}
\urladdr{https://www2.math.upenn.edu/~tbraz/}

\author{Stephen McKean}
\address{Department of Mathematics, Duke University} 
\email{mckean@math.duke.edu}
\urladdr{shmckean.github.io}
\begin{abstract}
    One can compute the local $\mathbb{A}^1$-degree at points with separable residue field by base changing, working rationally, and post-composing with the field trace. We show that for endomorphisms of the affine line, one can compute the local $\mathbb{A}^1$-degree at points with inseparable residue field by taking a suitable lift of the polynomial and transferring its local degree. We also discuss the general set-up and strategy in terms of the six functor formalism. As an application, we show that trace forms of number fields are local $\A^1$-degrees.
\end{abstract}
\maketitle
\section{Introduction}

Let $k$ be a field. In order to compute the local $\A^1$-Brouwer degree of a map $f:\A^n_k\to\A^n_k$ at a closed point $p$ with finite separable residue field $k(p)/k$, one can base change to the field of definition, compute the local degree of $f_{k(p)}$ at the canonical $k(p)$-rational point $\til{p}$ sitting over $p$, and then apply a field trace \cite{trace-paper}. That is, there is an equality
\begin{align*}
    \deg_p^{\A^1}(f) &= \Tr_{k(p)/k} \deg_{\til{p}}^{\A^1} (f_{k(p)})
\end{align*}
in the Grothendieck--Witt group $\GW(k)$. For general finite extensions, two issues arise when $k(p)/k$ is inseparable. First, the trace form of an inseparable extension is degenerate, so the field trace does not provide a well-defined transfer $\GW(k(p))\to\GW(k)$. While alternate transfers are available from motivic homotopy theory, the second issue is simply that base changing $f$ to $k(p)$ and applying a transfer yields a bilinear form whose rank is too large.

We rectify these issues by providing two new ways of lifting $f$. Assuming that $k(p)/k$ is a finite simple field extension with primitive element $t$, we consider two transfers arising from $\mathbb{A}^1$-homotopy theory, namely the \textit{geometric transfer}, denoted $\tau_k^{k(p)}(t)$, and the \textit{cohomological transfer}, denoted $\Tr_k^{k(p)}$. Some motivic yoga suggests that the local $\A^1$-degree of $f$ at $p$ is transferred down from the local degree of a suitable lift of $f$ at the $k(p)$-rational point $\til{p}$ (corresponding to the ideal $(x-t)$) above $p$. We introduce the \textit{geometric lift} $\fg$ and the \textit{cohomological lift} $\fc$ of our polynomial $f$ at the point $p$. In the separable setting, the cohomological lift agrees with the base change of $f_{k(p)}$, recovering the main result of \cite{trace-paper} in the univariate case.
\begin{theorem}\label{thm:main} 
Let $f: \A^1_k \to \A^1_k$ be a morphism with an isolated root at a closed point $p$. Then
\begin{align*}
    \deg_p^{\A^1}(f) = \tau_k^{k(p)}(t) \deg_{\til{p}}^{\A^1} (\fg) = \Tr_k^{k(p)} \deg_{\til{p}}^{\A^1} (\fc).
\end{align*}
\end{theorem}
The proof of \autoref{thm:main} will be given in \autoref{lem:geometric-lift} and \autoref{cor:coh-tr}. In \autoref{rem:unstable}, we discuss how a suitable definition of an \textit{unstable} transfer would imply that \autoref{thm:main} holds unstably. As a corollary of \autoref{thm:main}, we get an upper bound on the rank of the non-hyperbolic part of the local $\A^1$-degree of a polynomial map $f:\A^1_k\to\A^1_k$. 

\begin{corollary}\label{cor:bound on non-hyperbolic} 
Let $f:\A^1_k\to\A^1_k$ have a root at a closed point $p$, defined by a monic, irreducible polynomial $m(x)$ of some degree $n$. Let $t\in k(p)$ be a primitive element for the field extension $k(p)/k$. Then $f(x) = u(x) m(x)^d$ for some polynomial $u\not\in m(x)\cdot k[x]$, and
\begin{align*}
    \deg_p^{\A^1}(f) = \begin{cases} \frac{nd}{2} \mathbb{H} & d\text{ is even} \\ \frac{n(d-1)}{2} \mathbb{H} + \tau_k^{k(p)}(t) \left\langle u(t) \right\rangle & d\text{ is odd}. \end{cases}
\end{align*}
\end{corollary}
In particular, the rank of the non-hyperbolic part of $\deg_p^{\A^1}(f)$ is bounded above by $[k(p):k]$.

Another immediate corollary provides a connection between motivic degrees and scaled trace forms or scaled Scharlau forms.
\begin{corollary}[\textit{Scaled trace and Scharlau forms are $\A^1$-degrees}] 
Let $L/k$ be a finite, separable field extension with primitive element $t$. Then for any $\left\langle a \right\rangle\in \GW(L)$, the scaled trace form $\Tr_k^{L} \left\langle a \right\rangle$ and the scaled Scharlau form $\tau_k^{L}(t) \left\langle a \right\rangle$ are given by the local $\A^1$-degree of an endomorphism of $\A^1_k$.
\end{corollary}
Combined with the main result of \cite{BMP21}, this provides a method for computing scaled trace forms via B\'ezoutians.

\subsection{Outline}
We begin with some exposition on purity and the six functor formalism in \autoref{sec:purity}. In \autoref{sec:transfers}, we recall some basic material on transfers in stable motivic homotopy theory and give evidence suggesting the existence of lifts. We discuss relevant commutative and linear algebraic tools in \autoref{sec:algebra}. Finally, we define geometric and cohomological lifts of univariate polynomials, prove \autoref{thm:main}, and discuss applications to trace forms in \autoref{sec:lifts}.

\subsection{Acknowledgements}
We thank Tom Bachmann, Fr\'ed\'eric D\'eglise, Marc Hoyois, and Kirsten Wickelgren for their insightful comments about transfers, as well as David Harbater for helpful correspondence related to commutative algebra. The first named author is supported by an NSF Graduate Research
Fellowship (DGE-1845298). The second named author received support from Kirsten Wickelgren's NSF CAREER grant (DMS-1552730).

\section{Purity and the six functor formalism}\label{sec:purity}
In this section, we recall the six functor formalism in stable motivic homotopy theory. We also discuss the (previously established) reformulation of Morel--Voevodsky's purity theorem in terms of the six functors formalism. Throughout this section, we will assume that $k$ is a field finitely generated over a perfect field. This assumption will not be necessary when we arrive at our main results later in the paper.

Assigned to any scheme $X$, there is a stable symmetric monoidal category $\SH(X)$ of motivic spectra. Given any morphism $f: X \to Y$, there is an adjunction
\begin{align*}
    f^\ast : \SH(Y) \leftrightarrows \SH(X) : f_\ast,
\end{align*}
where $f^\ast$ is symmetric monoidal (in particular, it preserves sphere spectra: $f^\ast \mathbf{1}_Y = \mathbf{1}_X$). If $f$ is smooth, then $f^\ast$ admits a left adjoint, denoted $f_\sharp$, which is a ``forgetful'' functor. Finally, if $f$ is locally of finite type, then there is an exceptional adjunction
\begin{align*}
    f_! : \SH(X) \leftrightarrows \SH(Y) : f^!.
\end{align*}
When $f$ is a sufficiently nice morphism, many of these functors are isomorphic. If $f$ is proper, then there is a natural isomorphism $f_\ast \simeq f_!$, while if $f$ is \'etale, we have a natural isomorphism $f_! \simeq f_\sharp$. In particular, if $f$ is proper and \'etale, then $f_*\simeq f_\sharp$. In the case where $f$ is an open immersion, we have a natural isomorphism $f^\ast \simeq f^!$. Given a cartesian square, there are various exchange isomorphisms which allow one to interchange various six functors operations. Finally, we have a motivic $J$-homomorphism $K(X) \to \Pic(\SH(X))$ mapping any $\xi$ to $\Sigma^\xi \mathbf{1}_X$, where $\Sigma^\xi$ is the \textit{Thom transformation} associated to $\xi$. If $\xi$ is a vector bundle over $X$, then $\Sigma^\xi$ can be seen as smashing with the Thom space $\Th(\xi)$. We refer the reader to \cite[\S2]{deloop2} and \cite[\S4.1]{BW} for more about the six functor formalism.

We will use the following well-known result.
\begin{proposition}\label{prop:cofiber} 
Let $\pi: X \to S$ be a smooth $S$-scheme. Let $i: Z \hookto X$ be a closed immersion (not necessarily smooth over $S$). Then we have a canonical $\A^1$-homotopy equivalence in $\SH(X)$:
\begin{align*}
    \Sigma^\infty \frac{X}{X-Z} \simeq i_\ast \mathbf{1}_Z.
\end{align*}
\end{proposition}
\begin{proof} Denote by $j: X-Z \hookto X$ the open immersion of the complement of $Z$. The localization theorem (see \cite[Theorem 2.21, p. 114]{MV} and \cite[\S1]{Hoyois-localization}) then gives an exact sequence
\begin{align*}
    j_! j^! \to \id \to i_\ast i^\ast.
\end{align*}
As $j$ is an open immersion, we have that $j_! j^! \simeq j_\sharp j^\ast$. Applying this exact sequence at the sphere spectrum, we obtain
\begin{align*}
    j_\sharp j^\ast \mathbf{1}_X \to \mathbf{1}_X \to i_\ast i^\ast \mathbf{1}_X.
\end{align*}
We have that $j_\sharp j^\ast \mathbf{1}_X = j_\sharp \mathbf{1}_{X-Z}$, which is $\Sigma_+^\infty (X-Z)$ in $\SH(X)$. This implies that $i_\ast \mathbf{1}_Z$ is the cofiber of the natural inclusion $X-Z \hookto X$.
\end{proof}

\begin{definition}\cite[\S2.5]{DJK}
Let $f: X \to Y$ be a morphism that is smoothable, local complete intersection (lci), and separated of finite type. Let $\mathcal{L}_f$ be the cotangent complex of $f$. There is then a natural transformation
\begin{align*}
    \mathfrak{p}_f : \Sigma^{\mathcal{L}_f} f^\ast \to f^!,
\end{align*}
which is called the \textit{purity transformation}.
\end{definition}
If $f$ is smooth, then $\mf{p}_f$ is a natural isomorphism. While the purity transformation generally fails to be a natural isomorphism when $f$ is not smooth, some of its components may still be isomorphisms. That is, there may be spectra $E$ such that the map $\Sigma^{\mathcal{L}_f} f^\ast E \to f^! E$ is invertible. 

\begin{definition}\cite[Definition 4.3.7]{DJK}
A spectrum $E$ is called \textit{$f$-pure} if the component of purity $\Sigma^{\mc{L}_f}f^\ast E\to f^!E$ is invertible.
\end{definition}

\begin{proposition}\cite[Proposition 4.3.10]{DJK}\label{prop:f-pure}
Let $f: X \to Y$ be a smoothable, separated morphism of finite type between regular $k$-schemes. Assume that $E \in \SH(k)$ is a motivic spectrum pulled back from a motivic spectrum defined over a perfect subfield of $k$. Let $\pi: Y \to \Spec k$ denote the structure map. Then $f$ is lci and $\pi^\ast E$ is $f$-pure.
\end{proposition}

In particular, consider the map $q: \Spec k(p) \to \Spec k$ of regular $k$-schemes. This map satisfies the conditions of \autoref{prop:f-pure}, and since $\mathbf{1}_k$ is pulled back from any perfect subfield of $k$, the canonical purity morphism
\begin{equation}\label{eqn:purity-for-q}
\begin{aligned}
    \Sigma^{\mathcal{L}_q} q^\ast \mathbf{1}_k \xto{\sim} q^! \mathbf{1}_k
\end{aligned}
\end{equation}
is invertible. It is well-known that the purity isomorphism in \autoref{eqn:purity-for-q} subsumes the foundational theorem of Morel and Voevodsky \cite[Theorem 2.23, p. 115]{MV}. We will briefly discuss how to see this. Let $S$ be a scheme, and let $X$ and $Z$ be smooth $S$-schemes. Consider a (not necessarily smooth) closed immersion $i: Z \hookto X$:
\[ \begin{tikzcd}
    Z\rar[hook,"i" above]\ar[dr,"g" below left] & X\dar["f" right]\\
     & S.
\end{tikzcd} \]
From the short exact sequence $f^\ast \mathcal{L}_i \to \mathcal{L}_g \to \mathcal{L}_g$, we have the equality $\mathcal{L}_g = \mathcal{L}_{i\circ f} = i^\ast \mathcal{L}_f + \mathcal{L}_i$ in $K(Z)$. Let $\mathcal{N}_i$ be the normal bundle of $Z$ in $X$. Since $\mathcal{N}_i[1]=\mathcal{L}_i$ in $K(Z)$, we have
\begin{align*}
    \Sigma^{-i^\ast \mathcal{L}_f} \Sigma^{\mathcal{L}_g} g^\ast &= \Sigma^{\mathcal{L}_i} g^\ast = \Sigma^{- \mathcal{N}_i} g^\ast.
\end{align*}
We now apply purity to both $g$ and $f$ to obtain
\begin{align*}
    \Sigma^{-i^\ast \mathcal{L}_f} \Sigma^{\mathcal{L}_g} g^\ast &\cong \Sigma^{-i^\ast \mathcal{L}_f} g^! = \Sigma^{-i^\ast \mathcal{L}_f} i^! f^! \\
    &\cong i^! \Sigma^{-\mathcal{L}_f} f^! \cong i^! \Sigma^{- \mathcal{L}_f} \Sigma^{\mathcal{L}_f} f^\ast \cong i^! f^\ast.
\end{align*}
Thus we have a natural isomorphism $\Sigma^{- \mathcal{N}_i} g^\ast \cong i^! f^\ast$. Passing to left adjoints, we obtain $g_\sharp \Sigma^{\mathcal{N}_i} \cong f_\sharp i_\ast$. Finally, we consider the component of this equivalence at the sphere spectrum. By \autoref{prop:cofiber}, we have that $i_\ast \mathbf{1}_Z = \Sigma^\infty \frac{X}{X-Z}$ as $X$-motivic spectra. Forgetting along $f$ gives us $f_\sharp i_\ast\mathbf{1}_Z=f_\sharp\Sigma^\infty\frac{X}{X-Z}$ in $\SH(S)$. Conversely, we have that $\Sigma^{\mathcal{N}_i} \mathbf{1}_Z = \Th(\mathcal{N}_i)$ in $\SH(Z)$. Forgetting along $g$ gives us $g_\sharp\Th(\mathcal{N}_i)=g_\sharp\Sigma^{\mathcal{N}_i}\mathbf{1}_Z$ in $\SH(S)$. Since $g_\sharp\Sigma^{\mathcal{N}_i}\mathbf{1}_Z\simeq f_\sharp i_\ast\mathbf{1}_Z$, we have the following equivalence in $\SH(S)$:
\begin{align*}
    g_\sharp\Th \left( \mathcal{N}_i \right)\simeq f_\sharp\Sigma^\infty\tfrac{X}{X-Z}.
\end{align*}

\section{Transfers}\label{sec:transfers}
In this section we discuss transfers arising in stable motivic homotopy theory, as well as their algebraic incarnations for Grothendieck--Witt groups.

Given a finite simple extension, residue homomorphisms induce a transfer called the \textit{geometric transfer} \cite[\S 4.2]{Morel} arising in Milnor--Witt $K$-theory. The geometric transfer can alternatively be defined using motivic spaces. In an attempt to extend this definition to finite field extensions, one might naively factor a finitely generated field extension $k \subseteq L$ into a composite of simple field extensions, and then compose geometric transfers. However, such a composition of geometric transfers will depend on the choice of factorization, indicating that the geometric transfer is not functorial along arbitrary finite field extensions. This can be rectified by multiplication by a certain rank one bilinear form, built out of the choice of primitive element of the extension, yielding the \textit{cohomological transfer} \cite[\S 4.2]{Morel}. Alternatively, by incorporating all possible such factorizations simultaneously, one obtains a transfer along \textit{twisted} Grothendieck--Witt rings, called the \textit{absolute transfer} \cite[\S 5.1]{Morel}.

Throughout this section, we will maintain our assumption from \autoref{sec:purity} that $k$ is finitely generated over a perfect field, which allows us to align the absolute transfer with Gysin maps. This assumption can be dropped in latter sections where our main results are proved.

\subsection{Geometric transfers}
In Milnor $K$-theory, the \textit{residue homomorphisms} associated to discrete valuations enable the construction of transfers along field extensions. In Milnor--Witt $K$-theory, first defined by Hopkins and Morel, residue homomorphisms are still available, but ambiguities arise corresponding to a choice of uniformizing parameter. In degree zero, the Milnor--Witt $K$-theory of a field is the Grothendieck--Witt ring $\GW(k)$, so these residue homomorphisms permit us to define transfers of symmetric bilinear forms along finite simple field extensions.

Suppose that $p \in \A^1_k$ is a closed point, so that $k(p)/k$ is a finite simple field extension. Let $t\in k(p)$ be a primitive element of the extension with minimal polynomial $m(x)\in k[x]$. Considering the affine line as a subspace of the projective line with global sections $k(x)$, the minimal polynomial $m$ of $p\in \A^1_k \subseteq \P^1_k$ defines a discrete valuation on $k(x)$. With $m(x)$ as a uniformizing parameter, we obtain a residue homomorphism
\begin{align*}
    K_1^\MW(k(x)) \xto{\partial_p} \GW(k(p)).
\end{align*}
We additionally have a residue homomorphism  $-\partial_\infty: K_1^\MW(k(x)) \to \GW(k)$ for the point at infinity on the projective line, corresponding to the uniformizing parameter $-1/x$. Given a class $\alpha\in\GW(k(p))$, we may select an arbitrary preimage of $\alpha$ in $K_1^\MW(k(x))$ and then map to $\GW(k)$ along $-\partial_\infty$. It turns out that this defines a well-defined group homomorphism called the \textit{geometric transfer} \cite[\S4.2]{Morel}.

\begin{definition}\label{def:geom-transfer} 
The \textit{geometric transfer} for
 a finite simple extension $k(p)/k$ with primitive element $t$ is defined by
\begin{align*}
    \tau_k^{k(p)}(t) : \GW(k(p)) &\to \GW(k) \\
    \alpha &\mapsto -\del_\infty \left( \del_p^{-1}(\alpha) \right).
\end{align*}
\end{definition}

Turning our attention to motivic spaces, we can alternatively consider the composite of a collapse map and purity isomorphism to obtain a canonical map\footnote{Here we are using our assumption that $k$ is finitely generated over a perfect field in order to apply purity.}
\begin{align*}
    \P^1_k \to \frac{\P^1_k}{\P^1_k - p} \simeq \Th \left( \mathcal{N}_{p/\P^1_k} \right).
\end{align*}
The minimal polynomial of $p$ determines a non-canonical trivialization of the normal bundle, yielding an isomorphism $\Th\left( \mathcal{N}_{p/\P^1_k} \right) \simeq \Th \left( \mathcal{O}_{k(p)} \right)$. We now take cohomology (with coefficients in the Grothendieck--Witt sheaf) of the composite $\P^1_k \to \Th \left( \mathcal{O}_{k(p)} \right)$ to get a map $\GW(k(p)) \to \GW(k)$ that agrees with $\tau_k^{k(p)}(t)$. This geometric description motivates the terminology ``geometric transfer'' (see, for instance, \cite[\S4.2]{Morel}).

\begin{remark}
Note that any $k$-linear map $h:L \to k$ along a finite field extension will induce a transfer $h_\ast : \GW(L) \to \GW(k)$ by post-composition. It turns out that the geometric transfer is induced by a classical map called the \textit{Scharlau form}.
\end{remark} 

\begin{definition} 
Let $L/k$ be a finite simple extension with primitive element $t$. Then the \textit{Scharlau form} is the $k$-linear map $s:L \to k$ defined by
\begin{align*}
    s(t^j) &= \begin{cases} 1 & j=[L:k]-1, \\ 0 & \text{otherwise}. \end{cases}
\end{align*}
\end{definition}

\begin{lemma}\label{lem:geometric is scharlau}
\cite[Lemma 2.2]{Calmes-Fasel}, \cite[Lemma 5.10]{hoyois} Let $L/k$ be a finite simple extension with primitive element $t$, and let $s:L\to k$ be the Scharlau form associated to $t$. Then $\tau_k^{L}(t)=s_*$ as homomorphisms $\GW(L)\to \GW(k)$.
\end{lemma}

This description allows us to understand explicitly the geometric transfer of any rank one form in $\GW(L)$. We first set up some notation.

\begin{notation}\label{notn:rank 1 lemma}
Let $L/k$ be a finite simple extension of degree $n$ with primitive element $t$, so that $B_{L/k}:=\{1,t,\ldots,t^{n-1}\}$ is a $k$-vector space basis of $L$. Given an $L$-vector space $V$ with basis $B_{V/L}:=\{a_1,\ldots,a_d\}$, the set $B_{V/k}:=\{a_i,a_it,\ldots,a_it^{n-1}\}_{i=1}^d$ is a $k$-basis of $V$.
\end{notation}

\begin{lemma}\label{lem:scharlau-trace-gram-matx} 
In the context of \autoref{notn:rank 1 lemma}, let $\beta:V\times V\to L$ be a symmetric bilinear form whose Gram matrix with respect to $B_{V/L}$ is $(\beta_{ij})_{i,j}$, and let $s: L \to k$ be the Scharlau form. Then the Gram matrix of $s_\ast \beta$ with respect to $B_{V/k}$ is a block matrix whose $(i,j)\textsuperscript{th}$ block is equal to the Gram matrix of $s_\ast \left\langle \beta_{ij} \right\rangle$ with respect to $B_{L/k}$.
\end{lemma}
\begin{proof} 
By definition, the $(i,j)\textsuperscript{th}$ block of $s_*\beta$ in the basis $B_{V/k}$ is given by
\begin{align*}
\begin{tabular}{C | C C C C}
    & a_j & a_j t & \cdots & a_j t^{n-1} \\
    \hline
    a_i & s(\beta(a_i, a_j)) & s(\beta(a_i, a_j t)) & \cdots & s(\beta(a_i, a_j t^{n-1})) \\
    a_i t & s(\beta(a_i t, a_j)) & s(\beta(a_i t, a_j t)) & \cdots & s(\beta(a_i,t a_j t^{n-1})) \\ 
    \vdots & \vdots & \vdots & \ddots & \vdots\\
    a_i t^{n-1} & s(\beta(a_i t^{n-1},a_j)) & s(\beta(a_i t^{n-1}, a_j t)) & \cdots & s(\beta(a_i t^{n-1}, a_j t^{n-1}).
    \end{tabular}
\end{align*}
Since $t\in L$ and $\beta$ is $L$-bilinear, we can rewrite this block as
\begin{align*}
\begin{tabular}{C | C C C C}
    & a_j & a_j t & \cdots & a_j t^{n-1} \\
    \hline
    a_i & s(\beta_{ij}) & s(t\beta_{ij}) & \cdots & s( t^{n-1} \beta_{ij}) \\
    a_i t & s(t \beta_{ij}) & s(t^2 \beta_{ij}) & \cdots & s(t^n \beta_{ij}) \\ 
    \vdots & \vdots & \vdots & \reflectbox{$\ddots$} & \vdots\\
    a_i t^{n-1} & s (t^{n-1} \beta_{ij})& s(t^n \beta_{ij}) & \cdots & s(t^{2n-2} \beta_{ij}),
    \end{tabular}
\end{align*}
which is precisely $s_\ast \left\langle \beta_{ij} \right\rangle$ with respect to the $k$-basis $\{1,t,\ldots,t^{n-1}\}$ of $L$.
\end{proof}

\subsection{Cohomological transfers} 
Let $L/k$ be a finite simple extension with primitive element $t\in L$, and take $m(x)\in k[x]$ to be the minimal polynomial of $t$. Let $p$ be the \textit{exponential characteristic} of $k$, which is defined to be $\Char{k}$ in positive characteristic and 1 in characteristic 0. We may factor the extension $L/k$ as
\begin{align*}
    k \subset L_\text{sep} = k[t^{p^i}] \subseteq L,
\end{align*}
for some $i$, where $L_\text{sep}$ is the separable closure of $k$ in $L$. This implies that $m(x) = m_0(x^{p^i})$ for some suitable $m_0(x)\in k[x]$. Note that $m_0(x)$ is the minimal polynomial of $t^{p^i}$ over $k$, and hence is separable. Moreover, if $L/k$ is separable, then $m_0(x)=m(x)$.




\begin{notation}\label{nota:omega}
Using the notation from the previous paragraph, we define a distinguished polynomial $\omega_0(x)\in L[x]$ associated to the extension $L/k$ by
\begin{align*}
    \omega_0(x) := \frac{m_0(x)}{x-t^{p^i}}.
\end{align*}
Note that $t^{p^i}$ is a root of $m_0(x)$ since $t$ is a root of $m(x)$, so $\omega_0(x)$ is indeed a polynomial. Since $m_0(x)$ is separable, we see that $\omega_0(t) \in L^\times$. We will use $\omega_0(x)$ to define the cohomological transfer in terms of the geometric transfer in \autoref{def:cohom-lift}.
\end{notation}

\begin{example}\label{ex:omega-purely-inseparable} Let $L/k$ be a finite purely inseparable extension in characteristic $p$. Then its minimal polynomial is by definition of the form $x^{p^r}-a$ for some $a\in k$, so $m_0(x) = x-a$ and therefore $\omega_0(x) = 1$.
\end{example}

\begin{example}\label{ex:omega-separable} Let $L/k$ be a finite separable extension with primitive element $t$, and let $m(x)$ be the minimal polynomial of $t$. Then $\omega_0(x) = \frac{m(x)}{(x-t)}$, so $\omega_0(t) = m'(t)$ by the product rule.
\end{example}

\begin{definition}\label{def:cohomological-transfer} \cite[Definition 4.26]{Morel} For a finite simple extension $L/k$ with primitive element $t$, the \textit{cohomological transfer} $\Tr_k^L$ is defined to be the composite
\[ \begin{tikzcd}
    \GW(L)\rar["{ \left\langle \omega_0(t) \right\rangle }" above]\ar[dr,"\Tr_k^L" below left] & \GW(L)\dar["\tau_k^L(t)" right]\\
     & \GW(k).
\end{tikzcd} \]
\end{definition}

Under nice conditions, $\Tr^L_k$ does not depend on the choice of primitive element $t$ \cite[Theorem 4.27]{Morel}. Moreover, \textit{loc.~cit.}~also implies that the cohomological transfer is functorial along field extensions outside of characteristic two, so we can define the cohomological transfer of an arbitrary finite extension as the composite of cohomological transfers over constituent simple extensions. For finite separable extensions, the cohomological transfer recovers the transfer on Grothendieck--Witt groups induced by the field trace \cite[Lemma 2.3]{Calmes-Fasel}. For purely inseparable extensions, the cohomological transfer and the geometric transfer coincide by \autoref{ex:omega-purely-inseparable}.

\subsection{Absolute transfers}
Let $L/k$ be a finite, purely inseparable extension. We can factor this into simple extensions
\begin{align*}
    k = L_0 \subseteq L_1 \subseteq \cdots \subseteq L_n = L.
\end{align*}
Let $t_i$ be a primitive element for the simple extension $L_i/L_{i-1}$. From this we obtain a composite of geometric transfers
\begin{align*}
    \tau_k^{L_1}(t_1)\circ \tau_{L_1}^{L_2}(t_2) \circ \cdots \circ \tau_{L_{n-1}}^L(t_n) : \GW(L) \to \GW(k).
\end{align*}
This composite transfer is not independent of the tuple $(t_1, \ldots, t_n)$. However, this transfer depends only on the class of the element $dt_1 \wedge \cdots \wedge dt_n$ in the determinant of the $L$-vector space of K\"{a}hler differentials of $L$ over $k$ \cite[\S5]{Morel}. Thus any class in $\omega_{L/k} :=\det \Omega_{L/k}$ provides a way to transfer from $L$ down to $k$. This perspective allows us to produce a well-defined \textit{absolute transfer}
\begin{align*}
    \Tr_k^{k(p)}\left(\omega_{L/k}\right) : \GW(L, \omega_{L/k}) \to \GW(k),
\end{align*}
where $\GW(L, \omega_{L/k})$ denotes the twisted Grothendieck--Witt group \cite[Definition 5.4]{Morel}. In the simple setting, $\omega_{L/k}$ is a one-dimensional $L$-vector space, and therefore isomorphic to $L$, inducing a group isomorphism $\GW(L,\omega_{L/k}) \cong \GW(L)$. This idea can be leveraged to canonically \textit{untwist} the absolute transfer in odd characteristic to obtain a transfer $\GW(L) \to \GW(k)$, which coincides with the cohomological transfer \cite[Remark 5.6]{Morel}.

It turns out that the absolute transfer is hiding in the background of the definition of the local $\A^1$-Brouwer degree. We will establish this fact after recalling the definition of the local degree.

\begin{definition}
A point $p\in\A^n_k$ is called an \textit{isolated zero} of a morphism $f:\A^n_k\to\A^n_k$ is $f(p)=0$ and $p$ is isolated in its fiber $f^{-1}(0)$.
\end{definition}

Let $f: \A^n_k \to \A^n_k$ be a morphism of affine space, and let $p \in \A^n_k$ be an isolated zero of $f$. By viewing $\A^n_k \subseteq \P^n_k$ as a subscheme of projective space via a standard chart, $f$ induces a map
\begin{align*}
    \bar{f}: \frac{\P^n_k}{\P^n_k - p} \to \frac{\P^n_k}{\P^n_k - 0}\simeq \frac{\P^n_k}{\P^{n-1}_k}.
\end{align*}
Precomposing with the collapse map $c_p: \P^n_k/\P^{n-1}_k \to \P^n_k/(\P^n_k - p)$ yields a morphism $f_p$ as in the following diagram:
\[ \begin{tikzcd}
    \P^n_k/(\P^n_k - p) \rar["\bar{f}" above] & \P^n_k/\P^{n-1}_k\\
    \P^n_k/\P^{n-1}_k\uar["c_p" left]\ar[ur,"f_p" below right] & 
\end{tikzcd} \]
\begin{definition} 
Let $f:\A^n_k\to\A^n_k$, and let $p$ be an isolated zero of $f$. The \textit{local $\A^1$-degree} $\deg^{\A^1}_p(f)$ of $f$ at $p$ is the image of the homotopy class of $f_p$ under Morel's degree map 
\[\deg^{\A^1} : \left[ \P^n_k/\P^{n-1}_k, \P^n_k/\P^{n-1}_k \right]_{\SH(k)} \to \GW(k).\]
\end{definition}

If $k(p)/k$ is separable, then the stable class of the collapse map admits a tractable description \cite[Lemma~13]{KW-EKL}. In particular, since $\Spec k(p)$ is a smooth $k$-scheme, purity gives an equivalence
\begin{align*}
    \frac{\P^n_k}{\P^n_k -p}\simeq \Th T_p \P^n_k \simeq \left( \frac{\P^n_k}{\P^{n-1}_k} \right) \smashprod \Spec k(p)_+.
\end{align*}
From this, one can prove that the collapse map is $\left( \P^n_k/\P^{n-1}_k \right) \smashprod \eta_{1_k}$, where $\eta: \id \to q_\ast q^\ast$ is the unit of the pushforward-pullback adjunction for the structure map $q: \Spec k(p) \to \Spec k$.

If $k(p)/k$ is not separable, we defer to the theory of Gysin maps in order to characterize the collapse map.

\begin{proposition}\label{prop:cohomology-supported-on-a-point} 
Let $E \in \SH(k)$ be a motivic spectrum. Then the compactly supported cohomology of $\P^n_k$ on $p \in \P^n_k$ is given by
\begin{align*}
    E_p \left( \P^n_k \right) = E \left( \frac{\P^n_k}{\P^n_k - p} \right).
\end{align*}
\end{proposition}
\begin{proof} 
Let $i:\Spec{k(p)}\to\P^n_k$ be the closed immersion of $p$ into $\P^n_k$. Let $\pi:\P^n_k\to\Spec{k}$ be the structure map, which is smooth. Cohomology with compact supports (c.f. \cite[4.2.1]{BW}) is defined to be
\begin{align*}
    E_p \left( \P^n_k \right) := \left[ \mathbf{1}_k , \pi_\ast i_! i^! \pi^\ast E \right]_{\SH(k)}.
\end{align*}
Since $i$ is a closed immersion, it is a proper map, so we have a canonical natural isomorphism $i_! \simeq i_\ast$. As $\pi$ is smooth, $\pi_\sharp$ exists and is left adjoint to $\pi^\ast$. Combining these facts with the basic properties of adjunctions, we have a string of natural isomorphisms:
\begin{align*}
    \left[ \pi^\ast \mathbf{1}_k, i_! i^! \pi^\ast E \right]_{\SH(\P^n_k)} &\cong \left[ \mathbf{1}_{\P^n_k}, i_\ast i^! \pi^\ast E \right]_{\SH(\P^n_k)} &&\text{($i_*\simeq i_!$)}\\
    &\cong \left[ i^\ast\mathbf{1}_{\P^n_k}, i^! \pi^\ast E \right]_{\SH(k(p))} &&\text{($i^\ast$ left adjoint to $i_\ast$)} \\
    &\cong \left[ \mathbf{1}_{k(p)}, i^! \pi^\ast E \right]_{\SH(k(p))} &&\text{($i^\ast$ monoidal)}\\
    &\cong \left[ i_\ast \mathbf{1}_{k(p)}, \pi^\ast E \right]_{\SH(\P^n_k)} &&\text{($i_\ast\simeq i_!$ left adjoint to $i^!$)} \\
    &\cong\left[ \pi_\sharp i_\ast \mathbf{1}_{k(p)}, E \right]_{\SH(k)} &&\text{($\pi_\sharp$ left adjoint to $\pi^\ast$)}.
\end{align*}
\autoref{prop:cofiber} states that $i_\ast \mathbf{1}_{k(p)}$ is the cofiber $\P^n_k / \left( \P^n_k - p \right)$, while $\pi_\sharp$ is the forgetful functor. The result follows from the definition of $E\left(\frac{\P^n_k}{\P^n_k-p}\right)$.
\end{proof}

\begin{proposition}[\textit{The collapse map induces the Gysin transfer}]\label{prop:collapse-map-induces-abs-transfer}
Let $E$ be any motivic spectrum over $k$. Let $i: \Spec k(p) \to \P^n_k$ be the inclusion of a closed point $p$, and let $q: \Spec k(p) \to \Spec k$ denote the structure map. The collapse map $c_p : \frac{\P^n_k}{\P^{n-1}_k} \to \frac{\P^n_k}{\P^n_k - p}$ induces a map $c_p^\ast: E_p(\P^n_k,\O^n_k) \to E_{\A^n_k}(\P^n_k,\O^n_k)$, and the composite
\begin{align*}
    E\left(\Spec k(p), \mathcal{L}_{q} \right) \xto{i_!} E_p(\P^n_k, \O^n_k) \xto{c_p^\ast} E_{\A^n_k}(\P^n_k,\O^n_k) \simeq E(\Spec k)
\end{align*}
is equal to the Gysin transfer $q_!$ \cite[(2.2.4)]{deloop2}.
\end{proposition}
\begin{proof} This can be seen by the commutativity of the bottom rectangle of \cite[(3.2.12)]{deloop2}.
\end{proof}

Given a map $f:\A^n_k\to\A^n_k$ with an isolated zero $p$, the class $\bar{f}$ lives in the stable homotopy classes of maps from the cofiber $\P^n_k/\left( \P^n_k - p \right)$ into $\P^n_k / \P^{n-1}_k$. This group admits a nice algebraic description.

\begin{proposition}\label{prop:twisted-GW} There is an isomorphism of groups
\begin{align*}
    \left[ \frac{\P^n_k}{\P^n_k - p}, \frac{\P^n_k}{\P^n_k - 0} \right]_{\SH(k)} \cong \GW \left( k(p), \omega_q \right).
\end{align*}
\end{proposition}
\begin{proof}
Excision implies that
\begin{align*}
    \frac{\P^n_k}{\P^n_k - 0}\simeq \frac{\A^n_k}{\A^n_k - 0} \simeq \Th(\O_k^n),
\end{align*}
where $\O_k^n$ is the trivial rank $n$ bundle over a point. As an element of $\SH(k)$, we can write $\Sigma^\infty \Th(\O_k^n)$ as $\Sigma^n \mathbf{1}_k$. Let $\til{\pi}:\P^n_{k(p)}\to\Spec{k(p)}$, $\pi:\P^n_k\to\Spec{k}$, and $q:\Spec{k(p)}\to\Spec{k}$ be structure maps. Let $i:\Spec{k(p)}\to\P^n_k$ denote the inclusion of $p$, and let $\iota: \Spec k(p) \to \P^n_{k(p)}$ denote inclusion of the canonical $k(p)$-rational point $\til{p}$ lying over $p$. These maps fit into the commutative diagram
\begin{equation}\label{diag:commutative diagram}
\begin{tikzcd}
    \Spec k(p)\rar["\iota" above]\dar[equal] & \P^n_{k(p)}\dar\rar["\widetilde{\pi}" above]\pb & \Spec k(p)\dar["q" right] \\
    \Spec k(p)\rar["i" below] & \P^n_k\rar["\pi" below] & \Spec k.
\end{tikzcd}
\end{equation}
By \autoref{prop:cofiber} and purity, we can rewrite our mapping classes as
\begin{align*}
    \left[ \frac{\P^n_k}{\P^n_k - p}, \frac{\P^n_k}{\P^n_k - 0} \right]_{\SH(k)} &\cong \left[ \pi_\sharp i_\ast \mathbf{1}_{k(p)}, \Sigma^n \mathbf{1}_k \right]_{\SH(k)}.
\end{align*}
Functoriality implies $(q\til{\pi}\iota)_*=q_*\til{\pi}_*\iota_*$ and $(\pi i)_*=\pi_* i_*$. As $\pi$ and $\til{\pi}$ are both proper and \'etale, we have $\pi_*\simeq\pi_\sharp$ and $\til{\pi}_*\simeq\til{\pi}_\sharp$. Since $q$ is proper, we have $q_*\simeq q_!$, with right adjoint $q^!$. \autoref{diag:commutative diagram} thus allows us to rewrite
\begin{align*}
    \left[ \pi_\sharp i_\ast \mathbf{1}_{k(p)}, \Sigma^n \mathbf{1}_k \right]_{\SH(k)} & \cong \left[\pi_*i_*\mathbf{1}_{k(p)},\Sigma^n\mathbf{1}_k\right]_{\SH(k)} &&\text{($\pi_*\simeq\pi_\sharp$)}\\
    &\cong \left[q_*\til{\pi}_*\iota_*\mathbf{1}_{k(p)},\Sigma^n\mathbf{1}_k\right]_{\SH(k)} &&\text{($q\til{\pi}\iota=\pi i$)}\\
    &\cong\left[\til{\pi}_\sharp\iota_*\mathbf{1}_{k(p)},q^!\Sigma^n\mathbf{1}_k\right]_{\SH(k(p))} &&\text{($q^!$ right adjoint to $q_*$, $\til{\pi}_*\simeq\til{\pi}_\sharp$)}\\
    &\cong \left[ \Sigma^\infty \frac{\P^n_{k(p)}}{\P^n_{k(p)}- \til{p}}, q^! \Sigma^n \mathbf{1}_k \right]_{\SH(k(p))} &&\text{(\autoref{prop:cofiber})}\\
    &\cong \left[ \Sigma^n \mathbf{1}_{k(p)}, q^! \Sigma^n \mathbf{1}_k \right]_{\SH(k(p))}&&\text{(purity)}.
\end{align*}
We can now use the isomorphism $q^! \Sigma^n \cong \Sigma^n q^!$, desuspend, and remark that the sphere spectrum is $q$-pure (\autoref{eqn:purity-for-q}) to deduce
\[
    \left[ \Sigma^n \mathbf{1}_{k(p)}, \Sigma^n q^! \mathbf{1}_k \right] \cong \left[ \mathbf{1}_{k(p)}, q^! \mathbf{1}_k \right] \cong \left[ \mathbf{1}_{k(p)}, \Sigma^{\mathcal{L}_q} \mathbf{1}_{k(p)} \right].
\]
Since the unit map $\mathbf{1}_{k(p)} \to H\til{ \mathbb{Z}}$ induces an isomorphism on $\pi_0$, we have an isomorphism
\begin{align*}
    \left[ \mathbf{1}_{k(p)}, \Sigma^{\mathcal{L}_q} \mathbf{1}_{k(p)} \right] \cong \left[ \mathbf{1}_{k(p)}, \Sigma^{\mathcal{L}_q} H\til{ \mathbb{Z}} \right] = \til{\text{CH}}{}^0 \left( \Spec k(p), \omega_q \right).
\end{align*}
Here $\widetilde{\text{CH}}$ denotes the Chow--Witt groups of a scheme, which are represented by the motivic spectrum $H\til{ \mathbb{Z}}$, and $\omega_q=\det\mathcal{L}_q$. We conclude by noting that $\widetilde{\text{CH}}{}^0 \left( \Spec k(p), \omega_q \right) \cong \GW \left( \Spec k(p), \omega_q \right)$ (see e.g. \cite[p.~35]{deloop2}).
\end{proof}

\begin{corollary}[\textit{Precomposition with the collapse map is the absolute transfer}]\label{cor:collapse-map-abs-transfer}
The collapse map $c_p : \P^n_k/\P^{n-1}_k \to \P^n_k / \left( \P^n_k - p \right)$ induces a morphism
\begin{equation}\label{eqn:absolute transfer}
    \left[ \frac{\P^n_k}{\P^n_k - p}, \frac{\P^n_k}{\P^{n-1}_k} \right]_{\SH(k)} \to \left[ \frac{\P^n_k}{\P^{n-1}_k}, \frac{\P^n_k}{\P^{n-1}_k} \right]_{\SH(k)},
\end{equation}
which is a map of the form $\GW(k(p),\omega_q) \to \GW(k)$. This is the absolute transfer.
\end{corollary}
\begin{proof} 
By \autoref{prop:twisted-GW} and \cite[Corollary 1.24]{Morel}, \autoref{eqn:absolute transfer} can be written as a map $\GW(k(p),\omega_q)\to\GW(k)$. Taking $E=\mathbf{1}_k$ to be the sphere spectrum, \autoref{prop:collapse-map-induces-abs-transfer} implies that the collapse map induces a Gysin map $\GW(k(p),\omega_q) \to \GW(k)$. By \cite[Proposition 4.3.17]{deloop2}, the Gysin map coincides with the absolute transfer.
\end{proof}

\subsection{Hinting at lifts for transfers}
So far, we have discussed transfers in the context of both Grothendieck--Witt rings and motivic spectra. The following result suggests that one can lift the class $\bar{f}$ up to a class $\til{f}$ around the canonical $k(p)$-rational point $\til{p}$. We then ask if the lift $\til{f}$ is compatible with a given transfer $\tau$: is $\tau(\til{f})=\bar{f}$?

\begin{proposition}\label{prop:untwisting} 
Morel's canonical untwisting (in odd characteristic) can be thought of as a map of the form
\begin{align*}
    \left[ \frac{\P^n_k}{\P^n_k - p}, \frac{\P^n_k}{\P^n_k - 0} \right]_{\SH(k)} \cong \GW(k(p),\omega_q) \xto{\sim} \GW(k(p)) \cong \left[ \frac{\P^n_{k(p)}}{\P^n_{k(p)} - \widetilde{p}},\frac{\P^n_{k(p)}}{\P^n_{k(p)} - 0} \right]_{\SH(k(p))}.
\end{align*}
\end{proposition}
\begin{proof} 
Since both $\til{p}$ and $0$ are $k(p)$-rational, the equivalence $[ \frac{\P^n_{k(p)}}{\P^n_{k(p)}- \til{p}}, \frac{\P^n_{k(p)}}{\P^n_{k(p)} - 0}]_{\SH(k(p))} \cong \GW(k(p))$ follows immediately by purity and \cite[Corollary 1.24]{Morel}. The result now follows from \autoref{prop:twisted-GW}. 
\end{proof}

\begin{remark}
Suppose that $f$ is an endomorphism of $\A^n_k$ with an isolated root at a closed point $p$. This induces a class $\bar{f}\in\left[ \frac{\P^n_k}{\P^n_k - p}, \frac{\P^n_k}{\P^n_k - 0} \right]$ whose absolute transfer is $\deg^{\A^1}_p(f)$. However, \autoref{prop:untwisting} implies that we can untwist $\bar{f}$ to obtain a class $\til{f}\in\left[ \frac{\P^n_{k(p)}}{\P^n_{k(p)} - \til{p}}, \frac{\P^n_{k(p)}}{\P^n_{k(p)} - 0} \right]$ whose geometric transfer recovers $\deg_p^{\A^1}(f)$. This leads us to the question of lifts, transfers, and degrees: is there an endomorphism $g$ of $\A^n_{k(p)}$ such that $\bar{g}=\til{f}$? \autoref{lem:geometric-lift} answers this question in the affirmative in the univariate setting.
\end{remark}

\section{B\'ezoutians, Hankel forms, and Horner bases}\label{sec:algebra}
We now discuss a few algebraic tools used for computing local $\A^1$-degrees. The first tool will be the \textit{B\'ezoutian} $\Bez(f)$ of a map $f:\A^n_k\to\A^n_k$, which is a polynomial in $2n$ variables. The coefficients of $\Bez(f)$ determine a bilinear form $k[X_1,\ldots,X_n]/(f)\times k[Y_1,\ldots,Y_n]/(f)\to k$ whose isomorphism class is $\deg^{\A^1}(f)$ \cite{BMP21}. This was first noticed by Cazanave in the univariate case \cite{Cazanave}. One can also recover the local $\A^1$-degree $\deg_p^{\A^1}(f)$ from $\Bez(f)$ in a similar manner \cite{BMP21}.

The second tool will be \textit{Hankel} matrices. In the univariate case, the bilinear forms determined by B\'ezoutians have a particular structure (namely, they are represented by Hankel matrices). By exploiting this structure, one can easily diagonalize these bilinear forms to better understand their classes in $\GW(k)$.

The final tool will be \textit{Horner bases}, which serve as an alternative to the monomial basis of a quotient $k[x]/(f)$. We will also discuss how Horner bases interact with the Scharlau form when $k[x]/(f)$ is a field. This will be relevant in the proof of \autoref{lem:geometric-lift}.

\subsection{B\'ezoutians and $\A^1$-degrees}
Given a map $f/g:\P^1_k\to\P^1_k$, let
\begin{align*}
    \Bez(f/g):=\frac{f(X)g(Y) - f(Y)g(X)}{X-Y}\in k[X,Y]
\end{align*}
be its B\'ezoutian. Writing $\Bez(f/g)=\sum_{i,j}c_{ij}X^{i-1}Y^{j-1}$, the matrix of coefficients $\left( c_{ij} \right)$ defines the \textit{B\'ezoutian bilinear form} of $f/g$, and the class in $\GW(k)$ of this bilinear form recovers $\deg^{\A^1}(f/g)$ \cite{Cazanave}.

In the univariate case, every local $\A^1$-degree can be expressed as a global $\A^1$-degree of the projective line.

\begin{proposition}[\textit{Univariate local degrees are global degrees}]\label{prop:local-degrees-are-global-degrees} Let $f: \A^1_k \to \A^1_k$ be a map with an isolated zero at a closed point $p$, and let $m(x)\in k[x]$ be an irreducible polynomial that generates the maximal ideal corresponding to $p$. Then $f(x) = u(x)m(x)^d$ for some $u(x)\in k[x]$ that is nonvanishing at $p$, and
\begin{align*}
    \deg_p^{\A^1}(f) = \deg^{\A^1} \left( \P^1_k \xto{m^d/u} \P^1_k \right).
\end{align*}
\end{proposition}
\begin{proof} 
Combining Cazanave's theorem with \cite{BMP21}, it suffices to show that $\Bez(f)\equiv\Bez(m^d/u)\bmod{(f(X),f(Y))}$. Moreover, since $u(x)$ is not contained in the ideal $(m(x))$, we have an isomorphism
\begin{align*}
    \frac{k[x]_{(m)}}{(f)}\cong\frac{k[x]_{(m)}}{(m^d)}
\end{align*}
of $k$-algebras. It thus suffices to show that $\Bez(f)\equiv\Bez(m^d/u)\bmod{(m(X)^d,m(Y)^d)}$. We compute that
\begin{align*}
    \Bez(f) &= \frac{u(X)m(X)^d - u(Y) m(Y)^d}{X-Y} \\
    &= \frac{u(X)m(X)^d-u(Y)m(Y)^d}{X-Y}+\frac{u(Y)m(X)^d-u(Y)m(X)^d}{X-Y}\\
    &\quad\ +\frac{u(X)m(Y)^d-u(X)m(Y)^d}{X-Y}\\
    &= \frac{u(Y)m(X)^d - u(X)m(Y)^d}{X-Y} + \frac{u(X)-u(Y)}{X-Y}(m(X)^d+m(Y)^d) \\
    &\equiv \frac{u(Y)m(X)^d - u(X)m(Y)^d}{X-Y} \bmod{(m(X)^d,m(Y)^d)}.\\
    &\equiv \Bez(m^d/u)\bmod{(m(X)^d,m(Y)^d)}.&\qedhere
\end{align*}
\end{proof}

\begin{corollary}\label{cor:local-deg-of-min-poly-at-itself} 
Let $p\in\A^1_k$ be a closed point, and let $m(x)\in k[x]$ be an irreducible polynomial that generates the maximal ideal corresponding to $p$. This polynomial determines a map $m:\A^1_k\to\A^1_k$, and
\[\deg_p^{\A^1}(m)=\deg^{\A^1}(m)=\deg^{\A^1}(\P^1_k\xto{m/1}\P^1_k).\]
\end{corollary}
\begin{proof} 
The equality $\deg^{\A^1}(m)=\deg^{\A^1}(\P^1_k\xto{m/1}\P^1_k)$ is a special case of \autoref{prop:local-degrees-are-global-degrees} (with $d=1$ and $u=1$). Morally speaking, $\deg_p^{\A^1}(m)=\deg^{\A^1}(m)$ since $p$ is the only root of $m$ over $k$. More precisely, the isomorphism 
\begin{align*}
    \frac{k[x]}{(m)}\cong \frac{k[x]_{(m)}}{(m)}
\end{align*}
of $k$-algebras preserves the B\'ezoutian and basis of $k[x]/(m)$. By \cite[Lemma 4.7]{BMP21}, it follows that $\deg_p^{\A^1}(m)=\deg^{\A^1}(m)$.
\end{proof}

\subsection{Hankel and block Hankel forms}
A \textit{Hankel matrix} is a symmetric matrix with constant anti-diagonals. A symmetric bilinear form that can be represented by a Hankel matrix is called a \textit{Hankel form}. Hankel matrices and forms are classical objects of study \cite{Iohvidov}. In the univariate setting, we may observe that B\'ezoutian bilinear forms of polynomials can be naturally represented by Hankel matrices. As a motivating example, consider the polynomial $f(x) = x^3 + 3x^2 - 4x + 1$. Its B\'{e}zoutian is given by
\begin{align*}
    \Bez(f)= \frac{f(X) - f(Y)}{X-Y} = (X^2 + XY + Y^2) + 3(X+Y) - 4.
\end{align*}
Writing this in monomial basis for the global algebra $k[x]/f(x)$, we obtain
\begin{align*}
    \deg^{\A^1}(f) = \left(\begin{tabular}{C | C C C}
    & 1 & X & X^2 \\
    \hline
    1 & -4 & 3 & 1 \\
    Y & 3 & 1 & 0 \\
    Y^2 & 1 & 0 & 0
    \end{tabular}\right).
\end{align*}
In particular, $\deg^{\A^1}(f)$ is a Hankel form. Note that all the anti-diagonals below the main anti-diagonal are constantly zero. We call such a form an \textit{upper triangular} Hankel form. The isomorphism class in $\GW(k)$ of an upper triangular Hankel form is well-understood --- interestingly, none of the information lying above the main anti-diagonal matters.

\begin{proposition}\label{prop:upper-Hankel-form}\cite[Lemma~6]{KW} 
Let $s_1,\ldots,s_d\in k$ with $s_d \ne 0$. Then the matrix
\begin{align*}
    \begin{pmatrix} s_1 & s_2 & \cdots & s_{d-1} & s_d \\
    s_2 & s_3 & \cdots & s_d & 0 \\
    \vdots & \vdots & \reflectbox{$\ddots$} & \vdots & \vdots \\
    s_{d-1}& s_d & \cdots & 0 & 0 \\
    s_d & 0 & \cdots & 0 & 0\end{pmatrix}.
\end{align*}
represents the $\GW(k)$ class
\begin{align*}
    \begin{cases} \frac{d}{2} \mathbb{H} & d\text{ is even} \\ \frac{d-1}{2} \mathbb{H} + \left\langle s_d \right\rangle & d\text{ is odd}. \end{cases}
\end{align*}
\end{proposition}


The global $\A^1$-degree of any polynomial map $\P^1_k\to\P^1_k$ is an upper triangular Hankel form, so \autoref{prop:upper-Hankel-form} characterizes such $\A^1$-degrees. This characterization alternatively follows from the fact that any univariate polynomial can be na\"ively $\A^1$-homotoped to its leading term \cite[Example 2.4]{Cazanave}.

One might ask whether local $\A^1$-degrees of univariate polynomials exhibit a similar symmetry. Since localizing the global algebra $k[x]/(f)$ at a maximal ideal $m(x)\cdot k[x]$ (corresponding to an isolated zero $p$ of $f$) can decrease its rank, the monomials $\{1,x,\ldots,x^{\deg(f)}\}$ may not form a basis of $k[x]_{(m)}/(f)$. In a suitable basis, we will show that the Gram matrix of $\deg_p^{\A^1}(f)$ is a block upper triangular matrix with constant blocks on each anti-diagonal. We call such a form a \textit{block Hankel form}. We will also see that each block in this Gram matrix for $\deg_p^{\A^1}(f)$ is itself a Hankel matrix.\footnote{We have elected to not call this a \textit{Hankel block Hankel form}.}

As in \autoref{prop:upper-Hankel-form}, we will demonstrate that information above the main off-diagonal of blocks does not affect the $\GW(k)$ class of a block Hankel form. We first introduce some notation before proving this general result.

\begin{notation}\label{notn:block matrix}
Let $V$ be an algebra over a field $K$. Let
\[\mc{B}:=\left\{ a_1b_1,\ldots,a_1b_n,\ldots,a_db_1,\ldots,a_db_n \right\}\]
be a vector space basis for $V$. Let $\beta$ be a bilinear form on $V$. The $dn \times dn$ Gram matrix for $\beta$ in the basis $\mc{B}$ can be written as
\begin{align*}
    \beta_{\mc{B}} = \left(\begin{tabular}{C | C C C C} 
    & a_1 & a_2 & \cdots & a_d \\
    \hline
    a_1 & A_{11} & A_{12} & \cdots & A_{1d} \\
    a_2 & A_{21} & A_{22} & \cdots & A_{2d} \\
    \vdots & \vdots & \vdots & \ddots & \vdots \\
    a_n & A_{d1} & A_{d2} & \cdots & A_{dd} \end{tabular} \right),
\end{align*}
where each $A_{ij}$ is a block matrix of the form
\begin{align*}
    A_{ij} = \left(\begin{tabular}{C | C C C C} 
    & a_j b_1 & a_j b_2 & \cdots & a_j b_n \\
    \hline
    a_i b_1 & \beta_{ij}^{11} & \beta_{ij}^{12} & \cdots & \beta_{ij}^{1n} \\
    a_i b_2 & \beta_{ij}^{21} & \beta_{ij}^{22} & \cdots & \beta_{ij}^{2n} \\
    \vdots & \vdots & \vdots & \ddots & \vdots \\
    a_i b_n & \beta_{ij}^{n1} & \beta_{ij}^{n2} & \cdots & \beta_{ij}^{nn} \end{tabular} \right).
\end{align*}
That is, $\beta_{ij}^{\ell k}$ is the coefficient appearing on $a_i b_\ell \otimes a_j b_k$ in $\beta$.
\end{notation}

\begin{lemma}\label{lem:block-hankel-diag} Let $V$, $\mc{B}$, and $\beta$ be as in \autoref{notn:block matrix}. Assume that $\Char{K}\neq 2$. Suppose that $\beta$ is non-degenerate, and that $\beta_{\mc{B}}$ is a block Hankel matrix
\begin{align*}
    \beta_{\mc{B}} = \left(\begin{tabular}{C C C C C}
    A_1 & A_2 & \cdots & A_{d-1} & A_d \\
    A_2 & A_3 & \cdots & A_d & 0 \\
    \vdots & \vdots  & \reflectbox{$\ddots$} & \vdots & \vdots \\
    A_{d-1} & A_d & \cdots & 0 & 0 \\
    A_d & 0 & \cdots & 0 & 0
    \end{tabular} \right).
\end{align*}
Also suppose that each $A_i$ is an $n \times n$ Hankel matrix
\begin{align*}
    A_i = \begin{pmatrix} \beta_i^1 & \beta^2_i & \cdots & \beta^n_i \\
    \beta^2_i & \beta^3_i & \cdots & \beta^{n+1}_i \\
    \vdots & \vdots & \reflectbox{$\ddots$} & \vdots \\
    \beta^{n}_i & \beta^{n+1}_i & \cdots & \beta^{2n-1}_i
    \end{pmatrix}. 
\end{align*}
Then the class in $\GW(K)$ of $\beta$ is $\frac{nd}{2}\mb{H}$ if $d$ is even and $\frac{n(d-1)}{2}\mb{H}+\hat{A}_d$ if $d$ is odd.\footnote{Here we are abusing notation to conflate the Gram matrix $A_d$ with the isomorphism class of forms it represents in $\GW(k)$.}
\end{lemma}
\begin{proof} 
The goal here is to exhibit a basis $\mc{B}'$ such that the Gram matrix $\beta_{\mc{B}'}$ is block diagonal. In the basis $\mc{B}$, the Gram matrix for $\beta$ can be written as
\begin{align*}
    \beta_{\mc{B}} = \sum_{i,j=1}^d\sum_{\ell,k=1}^n \beta_{i+j-1}^{\ell+k-1} a_i b_\ell \otimes a_j b_k
\end{align*}
for some scalars $\beta_{i+j-1}^{\ell+k-1}\in K$. We will recursively use the rows of $\beta_{\mc{B}}$ to construct the basis $\mc{B}'$. See \autoref{sec:pictorial-intuition} for the intuition behind the following details. For $1 \le i \le \floor{\frac{d}{2}}$ and $1\le \ell \le n$, let
\begin{align*}
    \psi_i^\ell &= \frac{\beta_i^{2\ell-1}}{2} a_i b_{\ell} + \sum_{k=\ell+1}^n \beta_i^{2\ell-1+k} a_i b_k + \sum_{j=i+1}^d \sum_{k=1}^n \beta_j^{k+\ell-1} a_j b_k.
\end{align*}
Now let
\begin{align*}
    \mc{B}'&=\{a_1b_1,\psi_1^1,a_1b_2,\psi_1^2,\ldots,a_1b_n,\psi_1^n,\ldots,a_{\floor{d/2}}b_n,\psi_{\floor{d/2}}^n\}\\
    &\quad\ \cup\begin{cases} \emptyset & d\text{ is even} \\ \left\{ a_{\frac{d+1}{2}} b_1, \ldots, a_{\frac{d+1}{2}} b_n \right\} & d\text{ is odd}. \end{cases}
\end{align*}
The assumption that $\beta$ is non-degenerate implies that the elements of $\mc{B}'$ are linearly independent, so $\mc{B}'$ is a $K$-basis for $V$. We now rewrite $\beta_{\mc{B}}$ in terms of $\mc{B}'$:
\begin{align*}
    \beta_{\mc{B}} &= \sum_{i=1}^{\floor{d/2}} \sum_{\ell=1}^n \left( a_i b_\ell \otimes \psi_i^{\ell} + \psi_i^{\ell} \otimes a_i b_\ell \right) + \begin{cases} 0 & d\text{ is even} \\ \sum_{\ell,k=1}^n \beta_d^{\ell+k-1} a_{\frac{d+1}{2}} b_\ell \otimes a_{\frac{d+1}{2}} b_k & d\text{ is odd}. \end{cases}
\end{align*}
It follows that $\beta_{\mc{B}'}$ is block diagonal. For $1 \le i \le \floor{\frac{d}{2}}$, the $i\textsuperscript{th}$ block of $\beta_{\mc{B}'}$ (corresponding to the basis elements $\left\{ a_i b_1, \psi_i^1, \ldots, a_i b_n, \psi_i^n \right\}$) is
\[    \begin{tabular}{C | C C  C C  C  C C }
    & a_i b_1 & \psi_i^1 & a_i b_2 & \psi_i^2 & \cdots & a_i b_n & \psi_i^n \\
    \hline
    a_i b_1 & 0 & 1 &  &  &  &  &  \\
    \psi_i^1 & 1 & 0 &  &  &  &  &  \\
    a_i b_2 &  &  & 0 & 1 &  &  &  \\
    \psi_i^2 &  &  & 1 & 0 &  &  &  \\
    \vdots &  &  &  &  & \ddots &  &  \\
    a_i b_n &  &  &  &  &  & 0 & 1 \\
    \psi_i^n &  &  &  &  &  & 1 & \phantom{.}0.
\end{tabular} \]
This is a block sum of $n$ copies of the hyperbolic form $\mathbb{H}$. If $d$ is odd, the final block of $\beta_{\mc{B}'}$ (corresponding to the basis elements $\{a_{\frac{d+1}{2}}b_1,\ldots,a_{\frac{d+1}{2}}b_n\}$) is simply $A_d$. It follows that $\beta$ is the direct sum of hyperbolic forms, along with a direct summand of $\hat{A}_d$ when $d$ is odd.
\end{proof}

In \autoref{sec:lifts}, we will use \autoref{lem:block-hankel-diag} to compare the local $\A^1$-degree of a function $f$ with the transfer of the local $\A^1$-degree of the lift of $f$.

\subsection{Horner bases}
Many of our calculations in \autoref{sec:lifts} involve choosing convenient bases of quotients of polynomial rings. The \textit{Horner basis}, defined below, is a basis which is dual to the monomial basis with respect to the Scharlau form (see \autoref{prop:scharlau-horner}); this fact will be useful when we prove \autoref{lem:geometric-lift}. We will collect a few definitions and results from \cite{real-ag} for later use.

\begin{definition}\label{defn:Horner basis} \cite[Notation 8.6]{real-ag}
Let $m(x)=x^n+a_{n-1}x^{n-1}+\ldots+a_0\in k[x]$. Define the \textit{Horner polynomials}
\[\Hor_i(m,x):=\begin{cases}
1 & i=0,\\
x\Hor_{i-1}(m,x)+a_{n-i} & 1 \le i<n.
\end{cases}\]
The set $\{\Hor_{n-1}(m,x),\Hor_{n-2}(m,x),\ldots,\Hor_{0}(m,x)\}$ forms a $k$-basis of $k[x]/(m)$, which is called the \textit{Horner basis}.
\end{definition}

Let $s:k[x]/m(x)\to k$ be the Scharlau form associated to the primitive element $x$. The following proposition states that $s$ is a dualizing form for the monomial and Horner bases, in the sense of \cite[Definition~2.1]{BMP21}.

\begin{proposition}\label{prop:scharlau-horner} \cite[Proposition 9.18]{real-ag}
Let $0\leq i,j\leq n-1$. Then
\[s(x^i\Hor_{n-1-j}(m,x))=\begin{cases}1 & i=j,\\
0 & i\neq j.
\end{cases}\]
\end{proposition}
\begin{proof}
Since we have assumed that $m(x)$ is monic, the Kronecker form mentioned in \textit{loc.~cit.}~is equal to the Scharlau form.
\end{proof}

By \cite[Proposition~3.5(2)]{BMP21}, the Scharlau form gives a straightforward way to write down elements of $k[x]/m(x)$ in terms of the Horner basis. This is also proved directly in \cite[Corollary 9.19]{real-ag}.

\begin{corollary}\label{cor:horner-basis} 
For any $g\in k[x]/m(x)$, we have
\[g(x)\equiv\sum_{i=0}^{n-1}s(x^ig(x))\Hor_{n-1-i}(m,x)\bmod{(m(x))}.\]
\end{corollary}

We now show that there is a close connection between the B\'ezoutian of $m$ and the Horner basis associated to $m$. In the language of \cite[Definition 3.8]{BMP21}, we will demonstrate that the bilinear form induced by the Scharlau form is in fact a B\'ezoutian bilinear form.

\begin{proposition}\label{prop:Bezoutian-of-m} We have an equality in $k[X,Y]$ of the form
\begin{align*}
    \frac{m(X)-m(Y)}{X-Y} &= \sum_{i=0}^{n-1} X^i \Hor_{n-1-i}(m,Y).
\end{align*}
\end{proposition}
\begin{proof} 
Since $m(x)=\sum_{i=0}^n a_ix^i$ is a polynomial, its B\'ezoutian can be written as
\begin{align*}
    \frac{m(X)-m(Y)}{X-Y} &= \sum_{\ell=1}^n a_\ell \left( \sum_{i+j=\ell-1} X^i Y^j \right) = \sum_{i+j=0}^{n-1} a_{i+j+1} X^i Y^j.
\end{align*}
Next, the coefficient of $Y^j$ in $\Hor_{i}(m,Y)$ is $a_{n+j-i}$ when $i\geq j$, and is zero otherwise. In particular, the coefficient of $X^iY^j$ in $X^i\Hor_{n-1-i}(m,Y)$ is $a_{i+j+1}$.
Thus the coefficients of $X^iY^j$ in $\Bez(m)$ and $\sum_{i=0}^{n-1}X^i\Hor_{n-1-i}(m,Y)$ agree.
\end{proof}

To conclude this section, we will relate the coefficients of the B\'ezoutian in the Horner basis to the coefficients of the Scharlau transfer in the monomial basis. Since the Scharlau transfer is equal to the geometric transfer for finite simple extensions (\autoref{lem:geometric is scharlau}), the following result will be useful when computing a geometric transfer in \autoref{lem:geometric-lift}. See also \cite[Proposition 9.20]{real-ag}.

\begin{proposition}\label{prop:multiply-by-bez-to-get-geom-tr}
Let $L/k$ be a finite simple extension with primitive element $t$, and let $m(x)\in k[x]$ be the minimal polynomial of $t$. Given any $u(x)\in L[x]$, the coefficient matrix of $u(X)\frac{m(X)-m(Y)}{X-Y}$ in the Horner basis is equal to the coefficient matrix of $s_*\langle u(t)\rangle$ in the monomial basis.
\end{proposition}
\begin{proof} 
By \autoref{prop:Bezoutian-of-m}, we have
\begin{align*}
    \frac{m(X)-m(Y)}{X-Y} &= \sum_{i=0}^{n-1} X^i \Hor_{n-1-i}(m,Y).
\end{align*}
Multiplying both sides by $u(X)$, we obtain
\begin{equation}\label{eqn:uBez}
    u(X) \frac{m(X)-m(Y)}{X-Y} = \sum_{i=0}^{n-1} u(X) X^i \Hor_{n-1-i}(m,Y).
\end{equation}
Since $u(X)X^i = \sum_j s \left( u(X) X^{i+j} \right)\Hor_{n-1-j}(m,X)$ by \autoref{cor:horner-basis}, we can rewrite \autoref{eqn:uBez} as
\begin{align*}
    \sum_{i,j=0}^{n-1} s \left( u(X) X^{i+j} \right) \Hor_{n-1-i}(m,X) \Hor_{n-1-j}(m,Y).
\end{align*}
On the other hand, the coefficient matrix of $s_*\langle u(t)\rangle$ in the monomial basis is given by $\big(s(u(t)t^{i+j})\big)_{i,j=0}^{n-1}$. Since $k(t)=k[X]/m(X)$, we have $s(u(X)X^{i+j})=s(u(t)t^{i+j})\in k$, as desired.
\end{proof}

\section{Lifts of univariate maps and transfers of local degrees}\label{sec:lifts}
Given a map $f:\mb{A}^n_k\to\mb{A}^n_k$ with a non-rational isolated zero $p$, we would like to compute the local degree $\deg_p(f)\in\GW(k)$ by lifting $f$ to a map $\til{f}:\mb{A}^n_{k(p)}\to\mb{A}^n_{k(p)}$ with rational isolated zero $\til{p}$, computing $\deg_{\til{p}}(\til{f})\in\GW(k(p))$, and applying the appropriate transfer $\GW(k(p))\to\GW(k)$. If $k(p)/k$ is a finite, separable extension, one may take $\til{f}$ to be the base change $f_{k(p)}$ \cite{trace-paper}. However, if $k(p)/k$ is finite and purely inseparable, lifting $f$ to $f_{k(p)}$ yields a local degree whose rank is too large, as illustrated in \autoref{ex:rank too big}.

\begin{example}\label{ex:rank too big}
Let $k=\mb{F}_p(t)$ for some prime $p>2$, and let $f:\mb{A}^1_k\to\mb{A}^1_k$ be given by $f(x)=(x^p-t)^d$, where $d\geq 1$ is an integer. Take $q\in\mb{A}^1_k$ to be the non-rational point defined by the ideal $(x^p-t)\subset\mb{F}_p(t)[x]$, and note that $k(q)=\mb{F}_p(t^{1/p})$. Let $\til{q}=(x-t^{1/p})$ be the $k(q)$-rational lift of $q$. By \cite[p. 182]{SS75} (and e.g. \cite[Theorem 5.1]{BMP21}), we have
\begin{align*}
    \rank(\deg_q(f))&=\dim_k\frac{k[x]_q}{(f)},\\
    \rank(\deg_{\til{q}}(f_{k(q)}))&=\dim_{k(q)}\frac{k(q)[x]_{\til{q}}}{(f_{k(q)})}.
\end{align*}
Since $f$ is a polynomial of degree $pd$ lying in the maximal ideal $(x^p-t)$, we observe that $\dim_k k[x]_q/(f)=pd$. The freshman's dream implies $f_{k(q)}=(x-t^{1/p})^{pd}$, so it follows that $\dim_{k(q)}k(q)[x]_{\til{q}}/(f_{k(q)})=pd$ as well. Applying the geometric (equivalently, Scharlau) transfer $\tau^{k(q)}_{k}(t^{1/p})=s_*:\GW(k(q))\to\GW(k)$ scales rank by $[k(q):k]$, so
\[\rank(s_*\deg_{\til{q}}(f_{k(q)}))>\rank(\deg_q(f)).\]
This too-high rank issue arises from the splitting of the minimal polynomial $m(x)$ of $q$. Any morphism $f:\A^1_k\to\A^1_k$ vanishing at $q$ must be a multiple of $m$. If $k(q)/k$ is purely inseparable, then all linear factors of $m_{k(q)}$ are contained in the ideal $\til{q}$ and are hence not invertible in $k(q)[x]_{\til{q}}$. This stands in contrast with the separable case, where all but one linear factor of $m_{k(q)}$ are not contained in $\til{q}$ and are hence invertible in the relevant local ring. The invertibility of these factors of $m_{k(p)}$ causes the desired drop in dimension when constructing the quotient ring $k(q)[x]_{\til{q}}/(f_{k(q)})$.
\end{example}

As motivated by \autoref{prop:untwisting}, we would like to look for a suitable lift of $f$.

\begin{notation}\label{notn:lift section}
Throughout \autoref{sec:lifts}, let $p \in \A^1_k$ be a closed point with corresponding minimal polynomial $m(x)\in k[x]$. Since $\A^1_k=\Spec{k[x]}$, the residue field $L:=k(p)$ is a finite simple extension of $k$. Let $t$ be a primitive element of $L/k$. The canonical point $\til{p} \in \A^1_L$ is the point corresponding to the ideal $(x-t)\subset L[x]$. We fix $f(x) \in k[x]$ to be a polynomial vanishing at $p$, written uniquely as $f(x) = u(x)m(x)^d$, where $u(x)$ is not contained in the ideal corresponding to $p$ (that is, $u$ is non-vanishing at $p$).
\end{notation}

\subsection{Geometric lifts of univariate polynomials}
We now describe how to lift univariate polynomials relative to geometric and cohomological transfers. We begin with geometric lifts.

\begin{definition} 
Let $f(x) = u(x) m(x)^d$ and $p$ be as in \autoref{notn:lift section}. The \textit{geometric lift} of $f$ at the point $p$ is the polynomial $$\fg(x) := u(x) (x-t)^d\in L[x].$$
\end{definition}

Now that we have defined the geometric lift of $f$ at $p$, we can compute its local $\A^1$-degree.

\begin{lemma}\label{lem:local-deg-geom-lift} 
Let $f(x) = u(x) m(x)^d$ and $p$ be as in \autoref{notn:lift section}. Then, as elements of $\GW(L)$, we have
\begin{align*}
    \deg_{\til{p}}^{\A^1}(\fg) &=  \begin{cases}\frac{d}{2}\mathbb{H} & d\text{ is even} \\ \left\langle u(t) \right\rangle + \frac{d-1}{2}\mathbb{H} & d\text{ is odd}. \end{cases}
\end{align*}
\end{lemma}
\begin{proof} 
The B\'ezoutian of $\fg$ at $\til{p}$ will be an element of the algebra
\begin{align*}
    \frac{L[X]_{(X-t)}}{(u(X)(X-t)^d)} \otimes \frac{L[Y]_{(Y-t)}}{(u(Y)(Y-t)^d)} &\cong \frac{L[X]_{(X-t)}}{((X-t)^d)} \otimes \frac{L[Y]_{(Y-t)}}{((Y-t)^d)}.
\end{align*}
We expand the B\'{e}zoutian as
\begin{align*}
    \Bez(\fg) &= \frac{u(X) (X-t)^d - u(Y)(Y-t)^d}{X-Y}\\
    &=\frac{u(X)(X-t)^d-u(Y)(Y-t)^d}{X-Y}+\frac{u(X)(Y-t)^d-u(X)(Y-t)^d}{X-Y}\\
    &=u(X) \frac{(X-t)^d - (Y-t)^d}{(X-t)-(Y-t)}+\frac{u(X)-u(Y)}{X-Y}(Y-t)^d\\
    &\equiv u(X) \frac{(X-t)^d - (Y-t)^d}{(X-t)-(Y-t)}\bmod{((X-t)^d,(Y-t)^d)} \\
    &= u(X) \left( \sum_{i=0}^{d-1} (X-t)^i (Y-t)^{d-1-i} \right).
\end{align*}
Our next goal is to write $\Bez(\fg)$ with respect to the basis $\{(x-t)^{d-1},(x-t)^{d-2},\ldots,(x-t),1\}$ of $L[x]_{(x-t)}/((x-t)^d)$. In order to do so, we must expand $u(x)\mod (x-t)^d$ in this basis. This is done using a truncated Taylor series expansion. Let $u^{(i)}$ denote the $i\textsuperscript{th}$ Hasse derivative of $u(x)$. Then $\sum_{i=0}^{d-1}u^{(i)}(t)(x-t)^i\equiv u(x)\bmod{(x-t)^d}$, so
\begin{align*}
    \Bez(\fg) &= \left( \sum_{i=0}^{d-1} u^{(i)}(t)(X-t)^i \right)\left( \sum_{j=0}^{d-1} (X-t)^j (Y-t)^{d-1-j} \right).
\end{align*}
It follows that the B\'ezoutian bilinear form of $\fg$ with respect to the basis $\{(x-t)^{d-i}\}_{i=1}^d$ is
\begin{equation}\label{eqn:Bez-ftil-monomial}
\begin{aligned}
    \begin{tabular}{C | C C C C C }
& (X-t)^{d-1} & (X-t)^{d-2} & \cdots & (X-t) & 1 \\
\hline
(Y-t)^{d-1} & u^{(d-1)}(t) & u^{(d-2)}(t) & \cdots  & u^{(1)}(t) & u(t) \\
(Y-t)^{d-2} & u^{(d-2)}(t) & u^{(d-3)}(t) & \cdots  & u(t) & 0\\
\vdots & \vdots & \vdots & \reflectbox{$\ddots$} & \vdots & \vdots \\
(Y-t) & u^{(1)}(t) & u(t) & \cdots & 0 & 0\\
1 & u(t)& 0 & \cdots & 0 & \phantom{.}0.\\
\end{tabular}
\end{aligned}
\end{equation}
Since $u(x)$ is not an element of the maximal ideal $m(x)\cdot k[x]$, it cannot be an element of the maximal ideal $(x-t)\cdot L[x]$. In particular, $u(t)\neq 0$, so the result follows from \autoref{prop:upper-Hankel-form}.
\end{proof}

\autoref{cor:bound on non-hyperbolic} now follows from \autoref{lem:local-deg-geom-lift}.

\begin{proof}[Proof of \autoref{cor:bound on non-hyperbolic}]
Apply \autoref{lem:scharlau-trace-gram-matx} to \autoref{eqn:Bez-ftil-monomial}. Conclude with \autoref{lem:block-hankel-diag} to block diagonalize the bilinear form.
\end{proof}

Since we have computed $\deg^{\A^1}_{\til{p}}(\fg)$, we can compare its geometric transfer to $\deg^{\A^1}_p(f)$.

\begin{lemma}\label{lem:geometric-lift}
The geometric lift is compatible with the local degree and geometric transfer. That is, $\tau_k^{k(p)}(t) \left(\deg_{\til{p}}^{\A^1}(\fg)\right)=\deg_p^{\A^1}(f)$ in $\GW(k)$.
\end{lemma}
\begin{proof} 
Using the same idea as in the proof of \autoref{lem:local-deg-geom-lift}, we have
\begin{align*}
    \Bez(f) &= \frac{u(X) m(X)^d - u(Y)m(Y)^d}{X-Y} \\
    &\equiv u(X) \frac{m(X)^d - m(Y)^d}{X-Y}\bmod{(m(X)^d,m(Y)^d)}\\
    & = u(X) \frac{m(X)^d - m(Y)^d}{m(X) - m(Y)} \cdot\frac{m(X) - m(Y)}{X-Y}.
\end{align*}
For $0\leq j<n$, let $H_j(x):=\Hor_j(m,x)$ be the $j\textsuperscript{th}$ Horner polynomial associated to $m(x)$ (as defined in \autoref{defn:Horner basis}), and let 
\[\mc{B}_i(x)=\{H_{n-1}(x)m(x)^{d-1-i},H_{n-2}(x)m(x)^{d-1-i},\ldots,H_{0}(x)m(x)^{d-1-i}\}.\]
Note that $\mc{B}(x):=\bigcup_{i=0}^{ d-1}\mc{B}_i(x)$ is a $k$-basis of $k[x]_{(m)}/(f)\cong k[x]_{(m)}/(m^d)$, since all elements of this set have distinct polynomial degree. Collecting powers of $m(X)$ and $m(Y)$, we have
\begin{align*}
    \Bez(f) &\equiv u(X) \frac{m(X) - m(Y)}{X-Y} \left( \sum_{i=0}^{d-1} m(X)^i m(Y)^{d-1-i} \right)\bmod{(m(X)^d,m(Y)^d)}.
\end{align*}
In this expansion, each summand of $\Bez(f)$ is divisible by $m(X)^im(Y)^{d-1-i}$. In particular, in the basis $\mc{B}(X)\times\mc{B}(Y)$, the matrix of coefficients of $\Bez(f)$ is block upper left triangular, where the $(i,j)\textsuperscript{th}$ block corresponds to the coefficients of the basis elements $\mc{B}_i(X)\times\mc{B}_j(Y)$. By \autoref{lem:block-hankel-diag}, it suffices to compare the blocks of the coefficient matrix of $\Bez(f)$ along the main anti-diagonal to those appearing in $\tau^L_k(t)( \deg_{\til{p}}^{\A^1}(\fg))$. The blocks appearing along this diagonal consists of the coefficients of $u(X) \frac{m(X) - m(Y)}{X-Y}\bmod{(m(X),m(Y))}$ expanded in the Horner basis $\{H_{n-1}(X),\ldots,H_0(X)\}\times\{H_{n-1}(Y),\ldots,H_{0}(Y)\}$, because the coefficients of any terms of $u(X)\frac{m(X)-m(Y)}{X-Y}$ that are divisible by $m(X)$ or $m(Y)$ will be shifted to blocks above the main anti-diagonal. This is exactly the Gram matrix of $\tau_k^{L}(t) \langle u(t) \rangle$ (see \autoref{eqn:Bez-ftil-monomial}) by \autoref{prop:multiply-by-bez-to-get-geom-tr}. The desired result now follows from \autoref{lem:local-deg-geom-lift}.
\end{proof}

\begin{remark}[\textit{Unstable degree}]\label{rem:unstable}
We expect that \autoref{lem:geometric-lift} holds \textit{unstably}. While Morel's $\mb{A}^1$-degree homomorphism
\[
\deg^{\A^1}:[(\mb{P}^1_k)^{\wedge n},(\mb{P}^1_k)^{\wedge n}]_{\mc{H}_\bullet(k)}\to\GW(k)
\]
is an isomorphism for $n\geq 2$, this map is only an epimorphism for $n=1$ \cite{Morel}. Building on the work of Morel \cite[p. 1037]{Morel-ICM}, Cazanave showed that
\[(\deg^{\A^1},\det\Bez):[\mb{P}^1_k,\mb{P}^1_k]_{\mc{H}_\bullet(k)}\to\GW(k)\times_{k^\times/k^{\times 2}}k^\times\]
is an isomorphism \cite{Cazanave}, where $\Bez(f)$ is the B\'ezoutian bilinear form of the rational map $f$. Moreover, the $\mb{A}^1$-degree of $f$ is the isomorphism class of $\Bez(f)$, so $(\Bez,\det\Bez)$ can be regarded as the \textit{unstable $\mb{A}^1$-degree}.

In the proof of \autoref{lem:geometric-lift}, we showed that $\Bez_p(f)$ and $\tau_k^{k(p)}(t) \left(\Bez_{t}(\fg) \right)$ represent the same class in $\GW(k)$. However, we also showed that $\det\Bez_p(f)=(\det\Bez_{t}(\fg))^{[k(p):k]}$. Thus if the geometric transfer $\tau_k^{k(p)}(t):\GW(k(p))\to\GW(k)$ can be extended to an ``unstable transfer''
\[\left(\tau_k^{k(p)}(t),\phi\right):\GW(k(p))\times_{k(p)^\times/k(p)^{\times 2}}k(p)^\times\to\GW(k)\times_{k^\times/k^{\times 2}}k^\times\]
such that $\phi(a)=a^{[k(p):k]}$ for any $a\in k^\times$, then the geometric lift will be compatible with the unstable local degree and unstable transfer:
\[
    \left(\tau_k^{k(p)}(t)(\deg_{\til{p}}^{\A^1}(\fg)),\phi(\det\Bez_{\til{p}}(\fg))\right)=(\deg_p^{\A^1}(f),\det\Bez_p(f)).
\]
\end{remark}

\subsection{Cohomological lifts of univariate polynomials}
As discussed earlier, geometric transfers do not behave well with respect to composite field extensions. One can rectify this issue by twisting geometric transfers, which leads to the notion of cohomological transfers. In \autoref{lem:geometric-lift}, we saw that the geometric transfer of the local $\A^1$-degree at $\til{p}$ of the geometric lift of $f$ is the local degree of $f$ at $p$. Analogously, we will define the \textit{cohomological lift} of $f$ by twisting the geometric lift. We will also prove that the cohomological lift is compatible with the cohomological transfer.

\begin{definition}\label{def:cohom-lift} 
Let $f(x) = u(x) m(x)^d$ and $p$ be as in \autoref{notn:lift section}. The \textit{cohomological lift} of $f$ at $p$ is the polynomial $$\fc(x) := \omega_0(x)^d u(x)(x-t)^d\in L[x],$$ where $\omega_0(x)$ is the polynomial associated to the extension $L/k$ defined in \autoref{nota:omega}.
\end{definition}

\begin{corollary}\label{cor:coh-tr} 
The cohomological lift is compatible with the local $\A^1$-degree and cohomological transfer. That is, $\Tr_k^{k(p)} \deg_{\til{p}}^{\A^1}(\fc) = \deg_p^{\A^1}(f)$.
\end{corollary}
\begin{proof} 
Since $m_0(x)$ is a separable polynomial, $\omega_0(x)$ is non-vanishing at $t$. \autoref{lem:local-deg-geom-lift} thus implies that 
\begin{align*}
    \deg_{\til{p}}^{\A^1} (\fc) &= \begin{cases}\frac{d}{2}\mathbb{H} & d\text{ is even} \\ \left\langle \omega_0(t)^d u(t) \right\rangle + \frac{d}{2}\mathbb{H} & d\text{ is odd} \end{cases} \\
    &= \begin{cases}\frac{d}{2}\mathbb{H} & d\text{ is even} \\ \left\langle \omega_0(t) u(t) \right\rangle + \frac{d}{2}\mathbb{H} & d\text{ is odd} \end{cases} \\
    &=\left\langle \omega_0(t) \right\rangle \deg_{\til{p}}^{\A^1}(\fg).
\end{align*}
The result now follows from \autoref{def:cohomological-transfer} and \autoref{lem:geometric-lift}.
\end{proof}

\begin{proposition}\label{prop:separable-cohom-lift-is-base-change} 
Assume that $k(p)/k$ is separable. Then the cohomological lift of $f$ at $p$ is the base change $f_{k(p)}$.
\end{proposition}
\begin{proof} This follows from the observation that $\omega_0(x)(x-t) = m_0(x) = m(x)$ in this setting.
\end{proof}

For finite separable extensions, the cohomological transfer is equal to the field trace on Grothendieck--Witt groups \cite[Lemma 2.3]{Calmes-Fasel}. By \autoref{prop:separable-cohom-lift-is-base-change}, we have that \autoref{cor:coh-tr} recovers the main result of \cite{trace-paper} for univariate maps.

\begin{example} The cohomological lift and geometric lift of a polynomial agree at a point with purely inseparable residue field by \autoref{ex:omega-purely-inseparable}.
\end{example}

\begin{example} Consider the polynomial $f(x) = (x+2)(x-2)(x^2+1)^3 \in \R[x]$, vanishing at $(x^2+1)$. We have that the geometric lift of $f$ is
\begin{align*}
    \fg = (x+2)(x-2)(x-i)^3,
\end{align*}
while $\omega_0(x) = (x+i)$, so that $\fc(x) = f_\C(x)$.
\end{example}

\subsection{Trace forms and Scharlau forms}
Given a finite separable extension $L/k$, the \textit{trace form} $(x,y)\mapsto\Tr_{L/k}(xy)$ is an important invariant of the extension; see \cite{Conner-Perlis} for a survey. Post-composition with the field trace induces a homomorphism $\GW(L) \to \GW(k)$, which coincides with the cohomological transfer.

\begin{proposition}\label{prop:cohom-is-field-trace-separably} \cite[Lemma 2.3]{Calmes-Fasel}
Let $L/k$ be a finite separable field extension. Then post-composition with the field trace $\Tr_{L/k}: L \to k$ induces the cohomological transfer
\begin{align*}
    \Tr_k^L : \GW(L) &\to \GW(k) \\
    \left[ V \times V \xto{\beta} L \right] &\mapsto \left[ V \times V \xto{\beta} L \xto{\Tr_{L/k}} k \right].
\end{align*}
\end{proposition}

Similarly, associated to each $a\in L^\times$ is the \textit{scaled trace form} $(x,y)\mapsto\Tr_{L/k}(axy)$. Since the field trace induces the cohomological transfer for finite separable extensions, (scaled) trace forms are of the form $\Tr^L_k\langle a\rangle$. 

\begin{definition}
Let $L/k$ be a finite separable extension with primitive element $t$. Recall that the geometric transfer is equal to the Scharlau transfer (\autoref{lem:geometric is scharlau}). In analogy with (scaled) trace forms, we define the \textit{(scaled) Scharlau form} associated to $a\in L^\times$ as $\tau_k^L(t) \left\langle a \right\rangle$.
\end{definition}

We will show that the isomorphism class of any (scaled) trace form or Scharlau form along a finite separable field extension $L/k$ is given by a local $\mb{A}^1$-degree. Paired with the main result of \cite{BMP21}, we obtain a straightforward computational formula for the isomorphism class of any scaled trace form or Scharlau form in the separable setting. We first recall a result that allows us to relate cohomological and geometric transfers in the separable setting.

\begin{proposition}\label{prop:hoyois} \cite[Lemma 5.8]{hoyois} Let $L/k$ be a finite separable extension with primitive element $t$. Let $m(x)\in k[x]$ be the minimal polynomial of $t$. Then for any $\beta \in \GW(L)$, we have $\Tr_k^{L} \left( \beta \right) = \tau_k^{L}(t) \left( \left\langle m'(t) \right\rangle\cdot \beta \right)$.
\end{proposition}
\begin{proof}
Since $L/k$ is separable, we have $\omega_0(t)=m_0(t)=m'(t)$. The result thus follows from \autoref{def:cohomological-transfer}.
\end{proof}

After giving a definition, we will be ready to show that scaled trace forms are in fact local $\A^1$-degrees.

\begin{definition}\label{def:a(x)}
Let $L/k$ be a finite simple field extension with primitive element $t$. Given $a\in L$, we then have $a=\sum_{i=0}^{[L:k]-1}a_it^i$, with $a_i\in k$ uniquely determined (since $t$ is fixed). Define $a(x):=\sum_{i=0}^{[L:k]-1}a_ix^i\in k[x]$.
\end{definition}

\begin{proposition}[\textit{Scaled Scharlau forms are $\A^1$-degrees}]\label{prop:scaled-scharlau-forms}  
Let $L/k$ be a finite separable extension with primitive element $t$, and let $m(x)\in k[x]$ be the minimal polynomial of $t$. Let $p\in\A^1_k$ be the closed point defined by $m(x)$. Let $a \in L^\times$. Then
\begin{align*}
    \tau_k^{L}(t) \left\langle a \right\rangle = \deg_p^{\A^1} (a(x)m(x)).
\end{align*}
\end{proposition}
\begin{proof}
Let $h(x)=a(x)m(x)$. By \autoref{prop:separable-cohom-lift-is-base-change}, we have that $m(x) = \omega_0(x)(x-t)$. Since $a(x)$ is non-vanishing at $t$, the cohomological lift of $h(x)$ is simply the base change $h_{\mathfrak{c}}(x) = h_{L}(x)$. By \cite[Proposition 15]{KW-EKL}, its local degree at $t$ is
\begin{align*}
    \deg_t^{\A^1} (h_{L}) &= \left\langle \left.\frac{d}{dx} h_{L}(x) \right|_{x=t} \right\rangle \\
    &= \left\langle \left. a'(x)m(x) + a(x) m'(x) \right|_{ x=t } \right\rangle \\
    &= \left\langle a(t) m'(t) \right\rangle.
\end{align*}
Applying the cohomological transfer and invoking \autoref{cor:coh-tr}, we have $\deg_p^{\A^1}(h) = \Tr_k^{L} \deg_t^{\A^1} (h_\mathfrak{c})$. Combining this with \autoref{prop:hoyois} concludes the proof.
\end{proof}

\begin{example} 
The Scharlau form $\tau_k^{k(p)} \left\langle 1 \right\rangle$ is the local degree of the minimal polynomial of $p$ at the point $p$. This is also equal to the global degree of the minimal polynomial by \autoref{cor:local-deg-of-min-poly-at-itself}. This indicates that unscaled Scharlau forms are uninteresting, in the sense that they are either entirely hyperbolic or hyperbolic plus a summand of $\left\langle 1 \right\rangle$.
\end{example}

\begin{proposition}[\textit{Scaled trace forms are $\A^1$-degrees}]\label{prop:scaled-trace-forms}  
Let $L/k$ be a finite separable extension with primitive element $t$, and let $m(x)\in k[x]$ be the minimal polynomial of $t$. Let $a \in k(p)^\times$. Then
\begin{align*}
    \Tr_k^{L} \left\langle a \right\rangle = \deg_p^{\A^1} (a(x)m'(x)m(x)).
\end{align*}
\end{proposition}
\begin{proof} 
Let $h(x) = a(x) m'(x) m(x)$. The geometric lift is given by $h_\mathfrak{g}(x) = a(x) m'(x) (x-t)$, so the local degree of $h_\mathfrak{g}$ at $t$ is
\begin{align*}
    \deg_t^{\A^1}(h_\mathfrak{g}) &= \left\langle \left. \frac{d}{dx} a(x) m'(x)(x-t) \right|_{ x=t } \right\rangle \\
    &= \left\langle a(t) m'(t) \right\rangle.
\end{align*}
Combining this with \autoref{prop:hoyois}, we have that
\[
    \deg_p^{\A^1}(h) = \tau_k^{L}(t)\left( \deg_t^{\A^1} \left( h_\mathfrak{g} \right)\right) = \Tr_k^{L} \left\langle a \right\rangle.\qedhere
\]
\end{proof}

\begin{example} 
Let $K=\Q(\sqrt[3]{2})$ with minimal polynomial $m(x)=x^3-2$. The extension $K/\Q$ has trace form
\begin{align*}
    \Tr_{K/\Q}\langle 1\rangle&=\left(\Tr_{K/\Q}(\sqrt[3]{2}^i\cdot\sqrt[3]{2}^j)\right)_{0\leq i,j\leq 2}\\
    &=\begin{pmatrix}
    3 & 0 & 0\\
    0 & 0 & 6\\
    0 & 6 & 0
    \end{pmatrix}\\
    &=\langle 3\rangle+\mb{H}.
\end{align*}
Using the code provided in \cite{BMP-code}, we verify that $\deg_{\sqrt[3]{2}}^{\A^1}(m'(x)\cdot m(x))=\langle 3\rangle+\mb{H}$.
\end{example}

\begin{remark} 
Given any irreducible polynomial $m(x)\in k[x]$ (defining a finite simple field extension $L/k$) and any unit $a\in L^\times$, we can readily compute the scaled trace form $\Tr_k^{L}\left \langle a \right\rangle$ using \autoref{prop:scaled-trace-forms} together with the Sage code provided in \cite{BMP-code}.
\end{remark}

\appendix

\section{Pictorial intuition for diagonalization arguments}\label{sec:pictorial-intuition}

Suppose we are given a symmetric bilinear form that can be represented by an upper left triangular Hankel matrix. The intuition behind the proof of \autoref{prop:upper-Hankel-form} is that the data of the matrix can be repackaged into ``upper-left corners.'' To illustrate what we mean by this, consider the $5 \times 5$ example illustrated in \autoref{fig:hankel}.

\begin{figure}
\begin{tikzpicture}[baseline=-\the\dimexpr\fontdimen22\textfont2\relax ]
\matrix (m)[matrix of math nodes,left delimiter=(,right delimiter=)]
{
a_1 & a_2 & a_3 & a_4 & a_5\\
a_2 & a_3 & a_4 & a_5 &\phantom{a_5}\\
a_3 & a_4 & a_5 & &\\
a_4 & a_5 & & &\\
a_5 & \phantom{a_5} & & &\\
};

\begin{pgfonlayer}{myback}
\fhighlight[red]{m-1-1}{m-1-5}
\fhighlight[red]{m-1-1}{m-5-1}
\fhighlight[blue]{m-2-2}{m-2-5}
\fhighlight[blue]{m-2-2}{m-5-2}
\fhighlight[green]{m-3-3}{m-3-3}
\end{pgfonlayer}
\end{tikzpicture}
\caption{Upper triangular Hankel matrix}\label{fig:hankel}
\end{figure}

Given a vector space basis $\left\{x_1,\ldots,x_5\right\}$, this matrix defines a bilinear form by $$\sum_{i,j} a_{i+j-1} x_i\otimes x_j.$$ Consider all the terms with a factor of $x_1$, illustrated in red. The top row is given by $x_1\otimes(a_1 x_1 + \ldots + a_5 x_5)$; by symmetry, the first column is given by $(a_1x_1+\ldots+a_5x_4)\otimes x_1$. Note that $a_1 x_1^{\otimes 2}$ is double counted, so we define a new basis element
\begin{align*}
    \psi_1 = \frac{a_1}{2}x_1 + a_2 x_2 + \ldots + a_5 x_5.
\end{align*}
In this terminology, the first corner of the matrix (highlighted in red) can be rewritten as $x_1\otimes\psi_1 + \psi_1\otimes x_1$. Similarly, for the second corner (highlighted in blue), we can define
\begin{align*}
    \psi_2 &= \frac{a_3}{2}x_2 + a_4 x_3 + a_5 x_4.
\end{align*}
Then the blue portion of the form is $x_2\otimes\psi_2 + \psi_2\otimes x_2$. Finally, we are left with the lone term in green, which is $a_5 x_3^{\otimes 2}$. We can thus define a new basis $\left\{ x_1, \psi_1, x_2, \psi_2, x_3 \right\}$. In this basis, our form can be written as $$x_1\otimes\psi_1+\psi_1\otimes x_1+x_2\otimes\psi_2+\psi_2\otimes x_2+a_5 x_3^{\otimes 2},$$ so the isomorphism class of this form is $2 \mathbb{H} + \left\langle a_5 \right\rangle$. 

Note that the Hankel structure was not used in this discussion --- we only needed symmetry and upper left triangularity.

\begin{remark} 
The proof of \autoref{prop:upper-Hankel-form} holds when the matrix is symmetric and upper left triangular, so the Hankel assumption is unnecessary.
\end{remark}

Passing to a more general case, replace the each $a_i$ with a block matrix $A_i$ (see \autoref{fig:block hankel}). We will use the same idea to diagonalize this matrix. If there is an odd number of blocks along the diagonal, we will stop our modifications short of the central block.
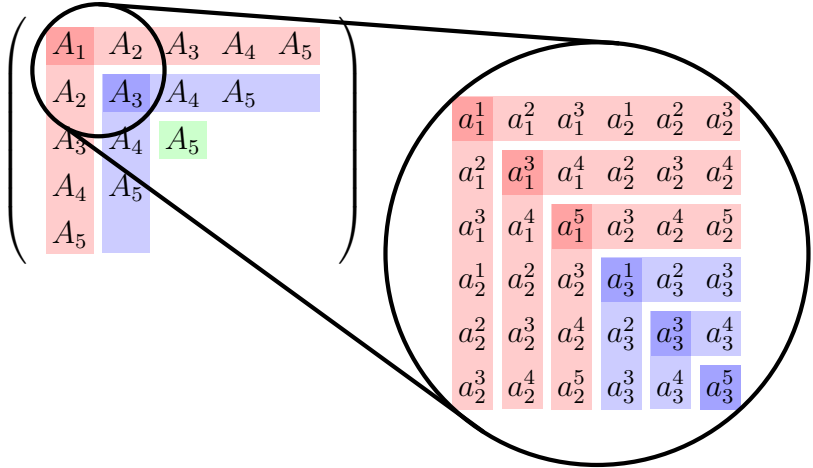
\begin{figure}
\begin{tikzpicture}[baseline=-\the\dimexpr\fontdimen22\textfont2\relax ]
\matrix (m)[matrix of math nodes,left delimiter=(,right delimiter=)]
{
A_1 & A_2 & A_3 & A_4 & A_5\\
A_2 & A_3 & A_4 & A_5 &\phantom{A_5}\\
A_3 & A_4 & A_5 & &\\
A_4 & A_5 & & &\\
A_5 & \phantom{A_5} & & &\\
};
\draw[ultra thick] (m-1-1.south east) circle (25pt);

\begin{scope}[xshift=5.5cm,yshift=-1.5cm]
\matrix (n)[matrix of math nodes]
{
a_1^1 & a_1^2 & a_1^3 & a_2^1 & a_2^2 & a_2^3\\
a_1^2 & a_1^3 & a_1^4 & a_2^2 & a_2^3 & a_2^4\\
a_1^3 & a_1^4 & a_1^5 & a_2^3 & a_2^4 & a_2^5\\
a_2^1 & a_2^2 & a_2^3 & a_3^1 & a_3^2 & a_3^3\\
a_2^2 & a_2^3 & a_2^4 & a_3^2 & a_3^3 & a_3^4\\
a_2^3 & a_2^4 & a_2^5 & a_3^3 & a_3^4 & a_3^5\\
};
\draw[ultra thick] (n-3-3.south east) circle (80pt);
\end{scope}

\draw[ultra thick] (m-1-1.south east)+(0,25pt) -- (5.8,1.29);
\draw[ultra thick] (m-1-1.south east)+(-12.5pt,-21.651pt) -- (4,-3.88);

\begin{pgfonlayer}{myback}
\fhighlight[red]{m-1-1}{m-1-5}
\fhighlight[red]{m-1-1}{m-5-1}
\fhighlight[blue]{m-2-2}{m-2-5}
\fhighlight[blue]{m-2-2}{m-5-2}
\fhighlight[green]{m-3-3}{m-3-3}
\fhighlight[red]{n-1-1}{n-1-6}
\fhighlight[red]{n-1-1}{n-6-1}
\fhighlight[red]{n-2-2}{n-2-6}
\fhighlight[red]{n-2-2}{n-6-2}
\fhighlight[red]{n-3-3}{n-3-6}
\fhighlight[red]{n-3-3}{n-6-3}
\fhighlight[blue]{n-4-4}{n-4-6}
\fhighlight[blue]{n-4-4}{n-6-4}
\fhighlight[blue]{n-5-5}{n-5-6}
\fhighlight[blue]{n-5-5}{n-6-5}
\fhighlight[blue]{n-6-6}{n-6-6}
\fhighlight[blue]{n-6-6}{n-6-6}
\end{pgfonlayer}
\end{tikzpicture}
\caption{Block upper triangular Hankel matrix}\label{fig:block hankel}
\end{figure}

We can now clarify the intuition behind the choice of
\begin{align*}
    \psi_i^\ell &= \underbrace{\vphantom{\sum_{j}^d}\frac{\beta_i^{2\ell-1}}{2} a_i b_{\ell}}_{\text{(i)}} + \underbrace{\vphantom{\sum_{j}^d}\sum_{k=\ell+1}^d \beta_i^{2\ell-1+k} a_i b_k}_{\text{(ii)}} + \underbrace{\vphantom{\sum_{j}^d}\sum_{j=i+1}^n \sum_{k=1}^d \beta_j^{k+\ell-1} a_j b_k}_{\text{(iii)}},
\end{align*}
which we used to diagonalize the block form in \autoref{lem:block-hankel-diag}. The term (i) is the term lying on the diagonal in the $i\textsuperscript{th}$ block on the $\ell\textsuperscript{th}$ row. The sum (ii) travels horizontally from the term on the diagonal until it reaches the edge of the block. Finally, the double sum (iii) continues the row to the right across all the other remaining blocks.

We can now decompose our form as a sum of hyperbolic forms $\sum_{i,\ell} a_ib_\ell\otimes\psi_i^\ell+\psi_i^\ell\otimes a_ib_\ell$. If there is an odd number of blocks, this decomposition will leave the central block (in this example, a copy of $A_5$) alone.

\begin{remark}
Again, we did not use any Hankel structure in this argument. In particular, the statement of \autoref{lem:block-hankel-diag} holds when the matrix is any symmetric matrix that is block upper left triangular.
\end{remark}

\bibliography{references}{}
\bibliographystyle{alpha}
\end{document}

%% file: autoref.tex
\usepackage{cleveref}
\usepackage{thmtools}

\def\makeautorefname#1#2{\expandafter\def\csname#1autorefname\endcsname{#2}}
\makeautorefname{eqn}{Equation}%
\makeautorefname{sec}{Section}%
\makeautorefname{subsec}{Subsection}%
\makeautorefname{footnote}{footnote}%
\makeautorefname{item}{item}%
\makeautorefname{figure}{Figure}%
\makeautorefname{table}{Table}%
\makeautorefname{part}{Part}%
\makeautorefname{app}{Appendix}%
\makeautorefname{assump}{assumption}%
\makeautorefname{cla}{claim}%
\makeautorefname{conj}{conjecture}%
\makeautorefname{cor}{corollary}%
\makeautorefname{cex}{counterexample}%
\makeautorefname{cexs}{counterexamples}%
\makeautorefname{dig}{digression}%
\makeautorefname{disc}{discussion}%
\makeautorefname{def}{definition}%
\makeautorefname{ex}{example}%
\makeautorefname{exs}{examples}%
\makeautorefname{fac}{fact}%
\makeautorefname{intu}{intuition}%
\makeautorefname{lem}{lemma}%
\makeautorefname{nota}{notation}%
\makeautorefname{note}{note}%
\makeautorefname{prop}{proposition}%
\makeautorefname{rmk}{remark}%
\makeautorefname{term}{terminology}%
\makeautorefname{thm}{theorem}%
\makeautorefname{upsh}{upshot}%
\makeautorefname{diag:}{Diagram}
\theoremstyle{definition}
\newtheorem{theorem}{Theorem}[section]

\newtheorem{corollary}[theorem]{Corollary}

\newtheorem{definition}[theorem]{Definition}

\newtheorem{example}[theorem]{Example}

\newtheorem{lemma}[theorem]{Lemma}
\newtheorem{notation}[theorem]{Notation}

\newtheorem{proposition}[theorem]{Proposition}
\newtheorem{remark}[theorem]{Remark}

\makeatletter
\let\c@corollary=\c@theorem
\let\c@proposition=\c@theorem
\let\c@lemma=\c@theorem
\let\c@conjecture=\c@theorem
\let\c@definition=\c@theorem
\let\c@example=\c@theorem
\let\c@remark=\c@theorem
\let\c@notation=\c@theorem
\makeatother